\documentclass[11pt]{article}

\usepackage{graphicx,amssymb,amsmath,amsthm,latexsym}
\usepackage[usenames]{color}

\usepackage{latexsym}
\usepackage{epic}
\usepackage{eepic}
\usepackage{amsfonts}

\newcommand{\reals}{\mathbb{R}}

\newtheorem{theorem}{Theorem}[section]
\newtheorem{lemma}[theorem]{Lemma}
\newtheorem{claim}[theorem]{Claim}

\newtheorem{corollary}[theorem]{Corollary}
\theoremstyle{definition}
\newtheorem{definition}[theorem]{Definition}

\newtheorem*{lemma-againa}{Lemma \ref{L-SM2}}
\newtheorem*{lemma-againb}{Lemma \ref{L-extreme}}

\newcommand{\ignore}[1]{}

\newcommand{\eq}{\;=\;}

\renewcommand{\d}{d}
\newcommand{\NN}{d}
\newcommand{\C}{C}

\newcommand{\mm}[1]{\mu_{\max}\{#1\}}
\newcommand{\MM}[1]{\mu_{\max}\left\{#1\right\}}

\newcommand{\on}[2]{\mathop{\null#2}\limits^{#1}}
\newcommand{\fvec}[1]{{\on{\,{}_\rightarrow}{{#1}}}}
\newcommand{\bvec}[1]{{\on{{}_\leftarrow}{{#1}}}}
\newcommand{\bfvec}[1]{{\on{\,{}_\leftrightarrow}{{#1}}}}

\newcommand{\M}{M}

\newcommand\mtcom[1]{}
\newcommand{\sgn}{\operatorname{sgn}}
\newcommand{\tint}[1]{\int_{t=-\frac12}^\frac12 {#1}\; \d t}

\begin{document}
\title{Maximum Overhang}

\author{
\em Mike Paterson
\thanks{Department of Computer Science, University of Warwick,
Coventry CV4 7AL, UK. E-mail: {\tt msp@dcs.warwick.ac.uk}}
\and \em Yuval Peres \thanks{Department of Statistics,
University of California, Berkeley, California 94710, USA.  E-mail:
{\tt peres@stat.berkeley.edu}} \and \em Mikkel Thorup \thanks{AT\&T
Labs - Research, 180 Park Avenue, Florham Park, NJ 07932, USA.
E-mail: {\tt mthorup@research.att.com}} \and \em Peter Winkler
\thanks{Department of Mathematics, Dartmouth College, Hanover,
NH 03755-3551, USA. E-mail: {\tt peter.winkler@dartmouth.edu}}
\and \em Uri Zwick \thanks{School of Computer Science, Tel Aviv
University, Tel Aviv 69978, Israel. E-mail: {\tt zwick@cs.tau.ac.il}} }


\maketitle

\begin{abstract}\noindent%
\setlength{\parindent}{0pt}%
\setlength{\parskip}{4pt plus 1pt}%
How far can a stack of $n$ identical blocks be made to hang over the
edge of a table?  The question dates back to at least the middle of
the 19th century and the answer to it was widely believed to be of
order $\log n$. Recently, Paterson and Zwick constructed $n$-block
stacks with overhangs of order $n^{1/3}$, exponentially better than
previously thought possible.
%
%
%
We show here that order $n^{1/3}$ is indeed best possible, resolving the
long-standing overhang problem up to a constant factor.
\end{abstract}

\section{Introduction} \label{sec:intro}

The problem of stacking~$n$ blocks on a table so as to achieve
maximum overhang has a long history. It appears in physics and
engineering textbooks from as early as the mid 19th century (see,
e.g., \cite{P50}, \cite{W55}, \cite{M07}). The problem was apparently
first brought to the attention of the mathematical community in 1923
when J.G. Coffin posed it in the ``Problems and Solutions" section
of the American Mathematical Monthly~\cite{C23}; no solution was
presented there.



\begin{figure}[here]
\begin{center}
\includegraphics[height=50mm]{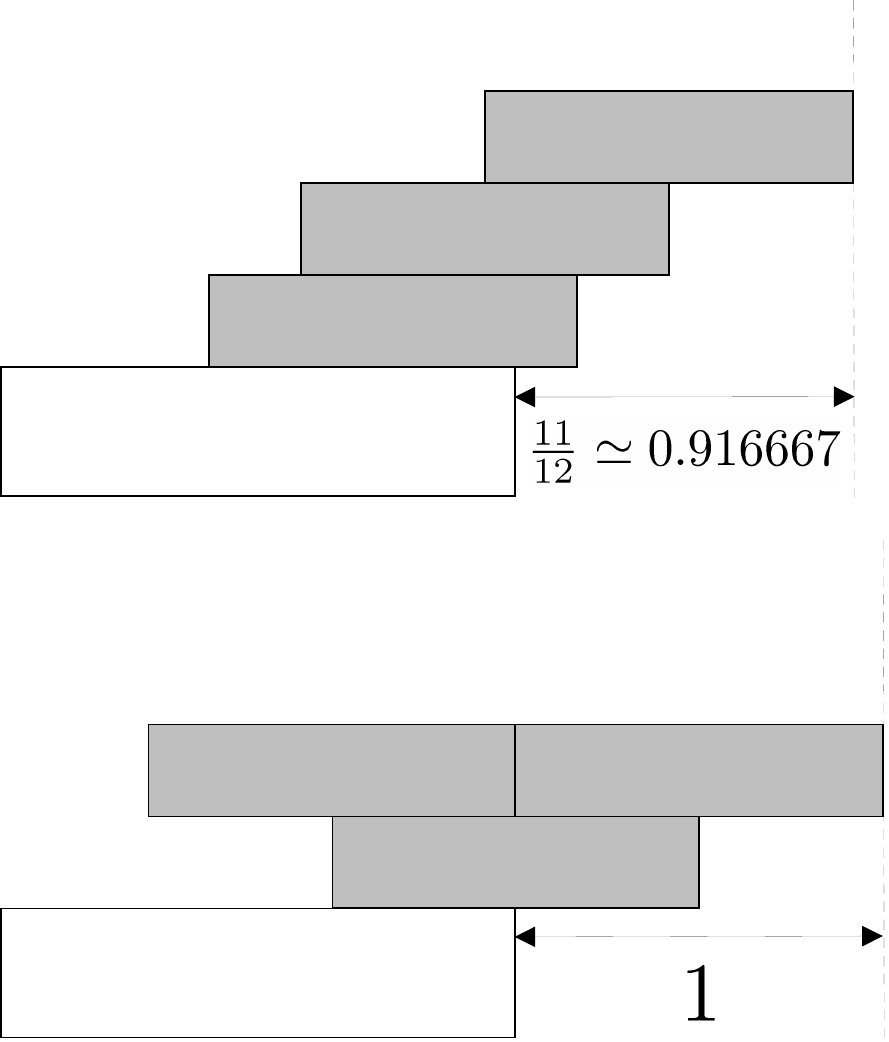}\hspace*{1cm}
\includegraphics[height=50mm]{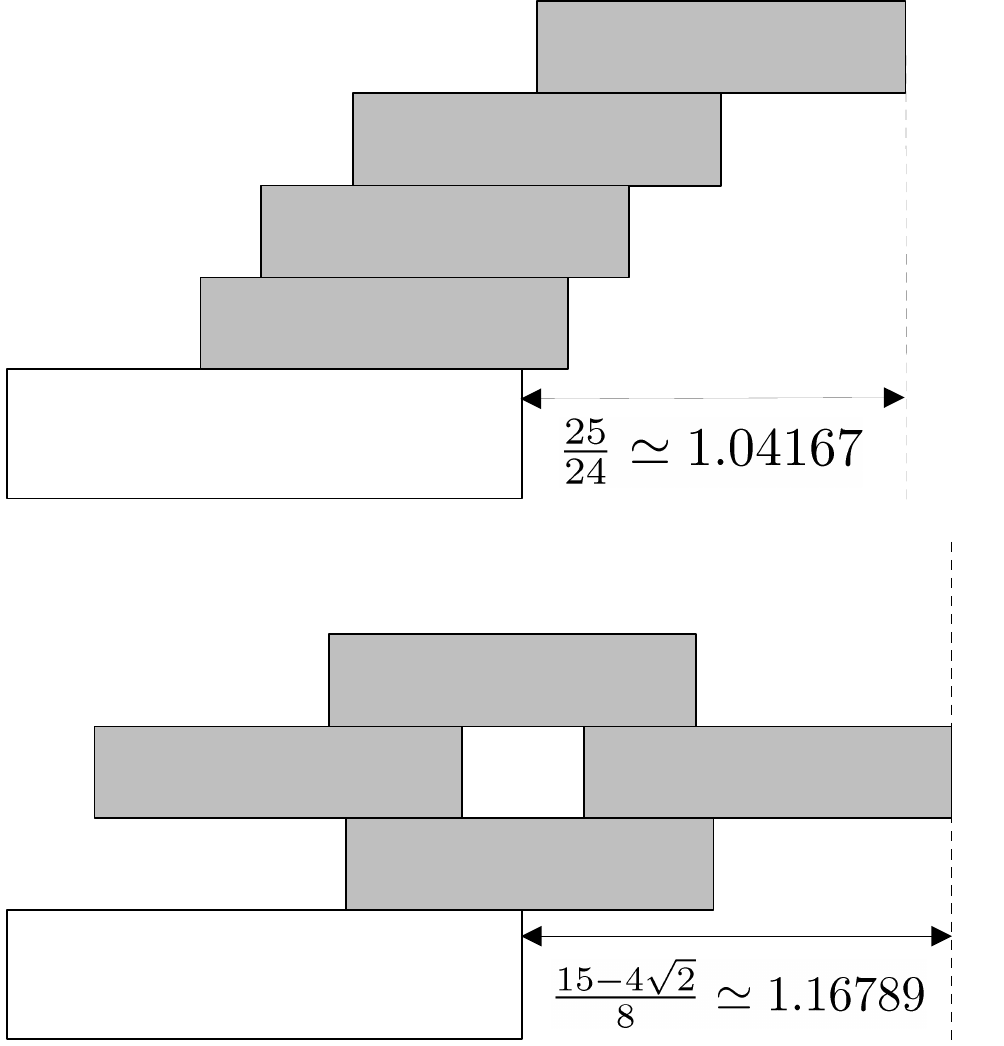}
\caption{Optimal stacks with 3 and 4 blocks, compared to the
corresponding harmonic stacks.
The 4 block solution is from \cite{A79}. Like
the harmonic stacks it can be made stable by minute displacements.} \label{fig:opt34}
\end{center}
\end{figure}

\begin{figure}[t]
\begin{center}
\includegraphics[width=80mm]{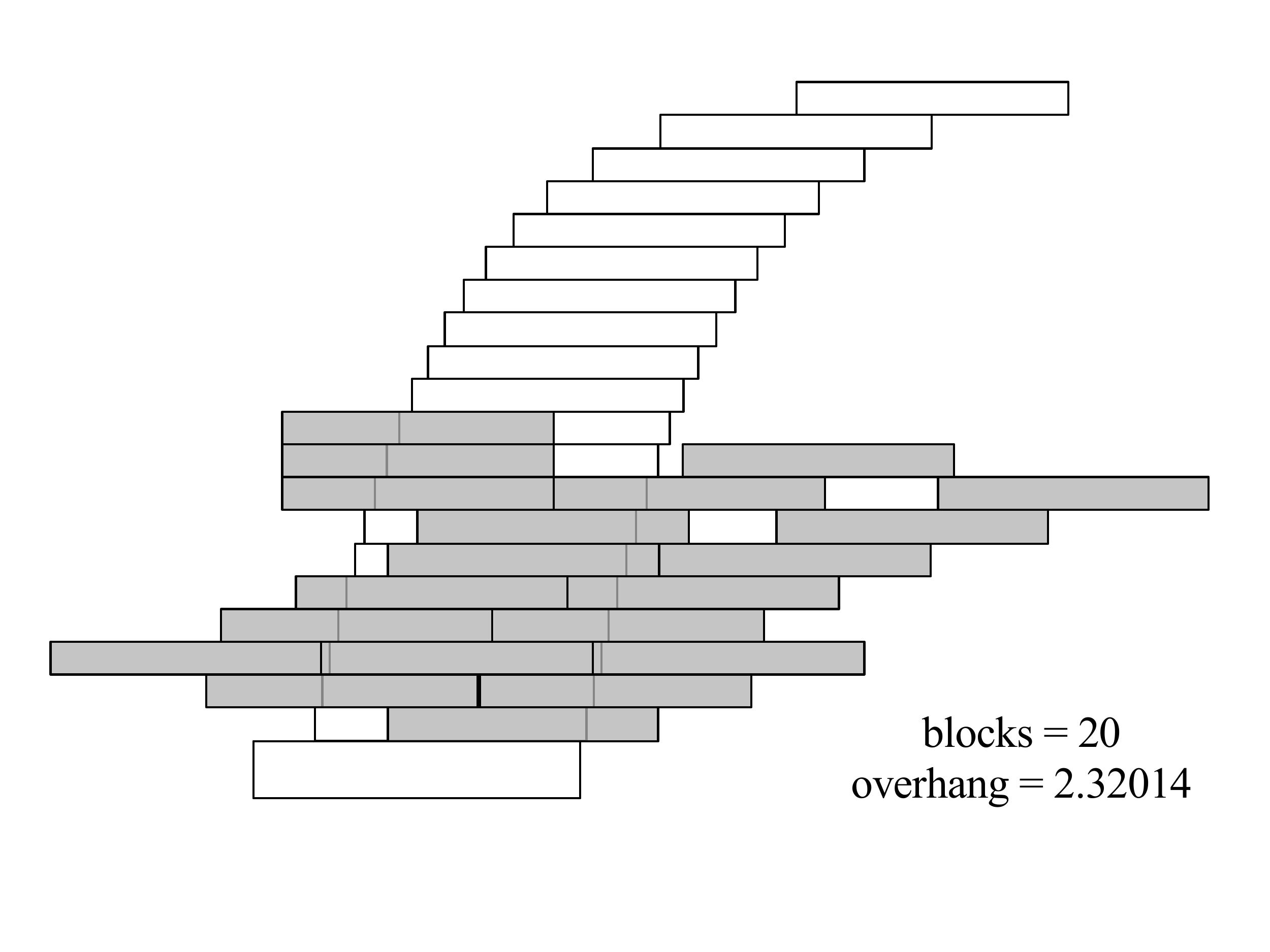}
\includegraphics[width=80mm]{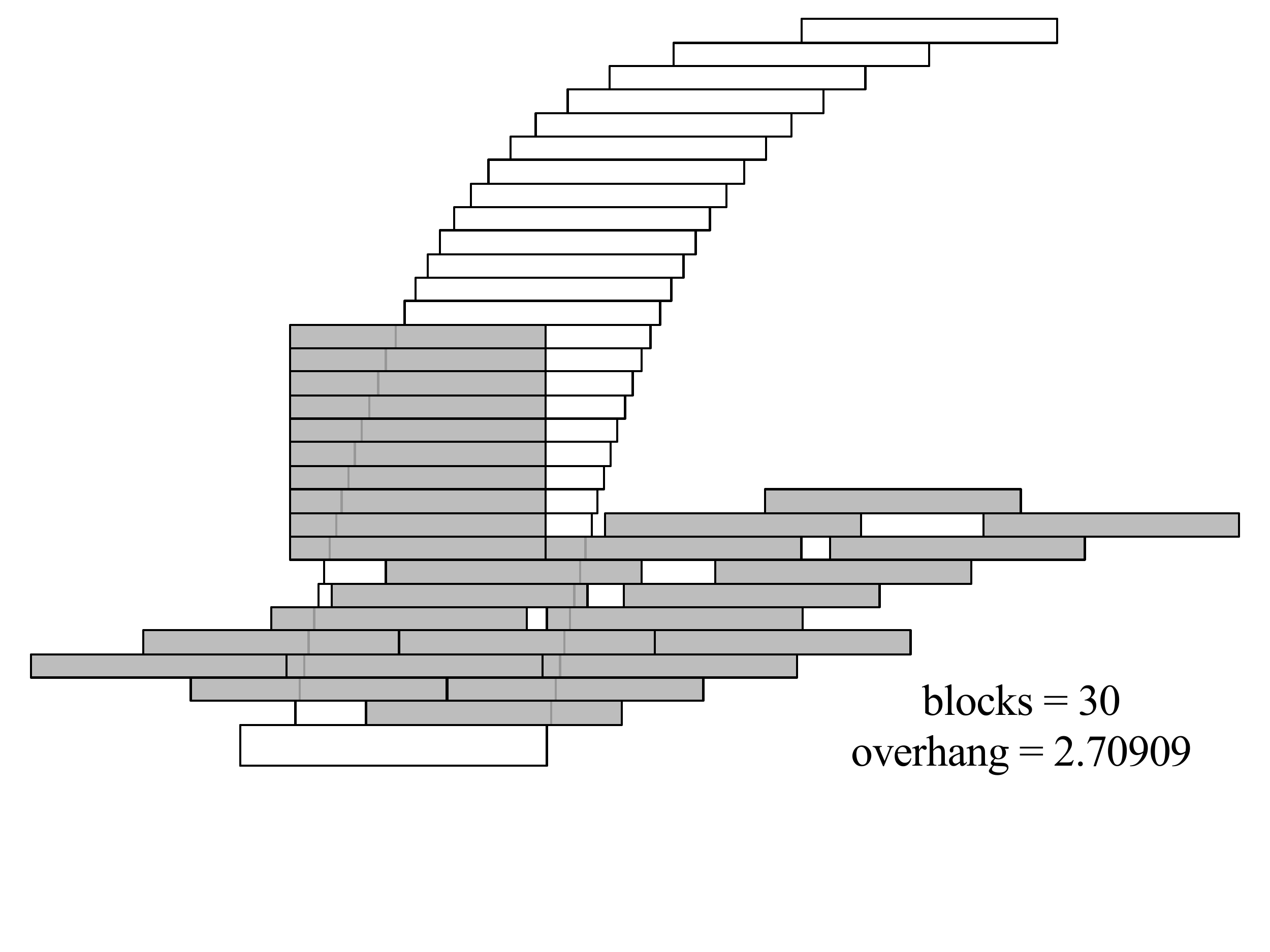}
\caption{Optimal stacks with 20 and 30 blocks
from~\cite{PZ06} with corresponding harmonic stacks in the background.} \label{fig:PZ}
\end{center}\vspace{-5mm}
\end{figure}


The problem recurred from time to time over subsequent years, e.g.,
\cite{S53, S54, S55, J55, GS58, E59, G64, G71, A79, D81, GKP88,
H05}, achieving much added notoriety from its appearance in 1964 in
Martin Gardner's ``Mathematical Games" column of \emph{Scientific
American} \cite{G64,G71}.

Most of the references mentioned above describe the now classical
\emph{harmonic stacks} in which $n$ unit-length blocks are placed
one on top of the other, with the $i^{\rm th}$ block from the top
extending by $\frac{1}{2i}$ beyond the block below it. The overhang
achieved by such stacks is $\frac12
H_n=\frac12\sum_{i=1}^n\frac{1}{i}\sim \frac12\ln n$. The cases
$n=3$ and~$n=4$ are illustrated at the top of Figure~\ref{fig:opt34}
above, and the cases $n=20$ and $n=30$ are shown in the background
of Figure~\ref{fig:PZ}.
Verifying that harmonic stacks are \emph{balanced} and can be made
\emph{stable} (see definitions in the next section) by minute
displacements is an easy exercise. (This is the form in which
the problem appears in \cite{P50}, \cite{W55}, \cite{M07}.) Harmonic
stacks show that arbitrarily large overhangs can be
achieved if sufficiently many blocks are available. They have been
used by countless teachers as an introduction to recurrence
relations, the harmonic series and simple optimization problems
(see, e.g., \cite{GKP88}).

\subsection{How far can you go?}
Many readers of the above mentioned references were led to believe
that $\frac12 H_n (\sim \frac12\ln n)$, the overhang achieved by
harmonic stacks, is the \emph{maximum} overhang that can be achieved
using~$n$ blocks. This is indeed the case under the restriction,
explicit or implicit in some of these references, that the blocks
should be stacked in a \emph{one-on-one} fashion, with each block
resting on at most one other block. It has been known for some time,
however, that larger overhangs may be obtained if the one-on-one
restriction is lifted. Three blocks, for example, can easily be used
to obtain an overhang of~1.
Ainley \cite{A79} found that four blocks can
be used to obtained an overhang of about 1.16789, as shown at the
bottom right of Figure~\ref{fig:opt34}, and this is more than 10\% larger
than the overhang of the corresponding
harmonic stack.
Using computers, Paterson and Zwick \cite{PZ06} found the optimal
stacks with a given limited number of blocks. Their solutions
with 20 and 30 blocks are shown in Figure~\ref{fig:PZ}.


Now what happens when $n$ grows large? Can general stacks, not
subject to the one-on-one restriction, improve upon the overhang
achieved by the harmonic stacks by more than a constant factor, or
is overhang of order $\log n$ the best that can be achieved? In a
recent cover article in the \emph{American Journal of Physics}, Hall
\cite{H05} observes that the addition of counterbalancing blocks to
one-on-one stacks can double (asymptotically) the overhang
obtainable by harmonic stacks.  However, he then incorrectly
concludes that no further improvement is possible, thus perpetuating
the order $\log n$ ``mythology''.





Recently, however, Paterson and Zwick \cite{PZ06} discovered that
the modest improvements gained for small values of~$n$ by using
layers with multiple blocks mushroom into an exponential improvement
for large values of~$n$, yielding overhang of order $n^{1/3}$
instead of just $\log n$.

\subsection{Can we go further?}
But is $n^{1/3}$ the right answer, or is it just the start of
another mythology? In their deservedly popular book {\em Mad About
Physics} \cite{JP01}, Jargodzki and Potter rashly claim that
inverted triangles (such as the one shown on the left of
Figure~\ref{fig:pyramid}) are balanced.  If so, they would achieve
overhangs of order $n^{1/2}$.
It turns out, however, that already the $3$-row inverted triangle is
unbalanced, and collapses as shown on the right of
Figure~\ref{fig:pyramid}, as do all larger inverted triangles.

\begin{figure}[t]
\begin{center}
\includegraphics[width=50mm]{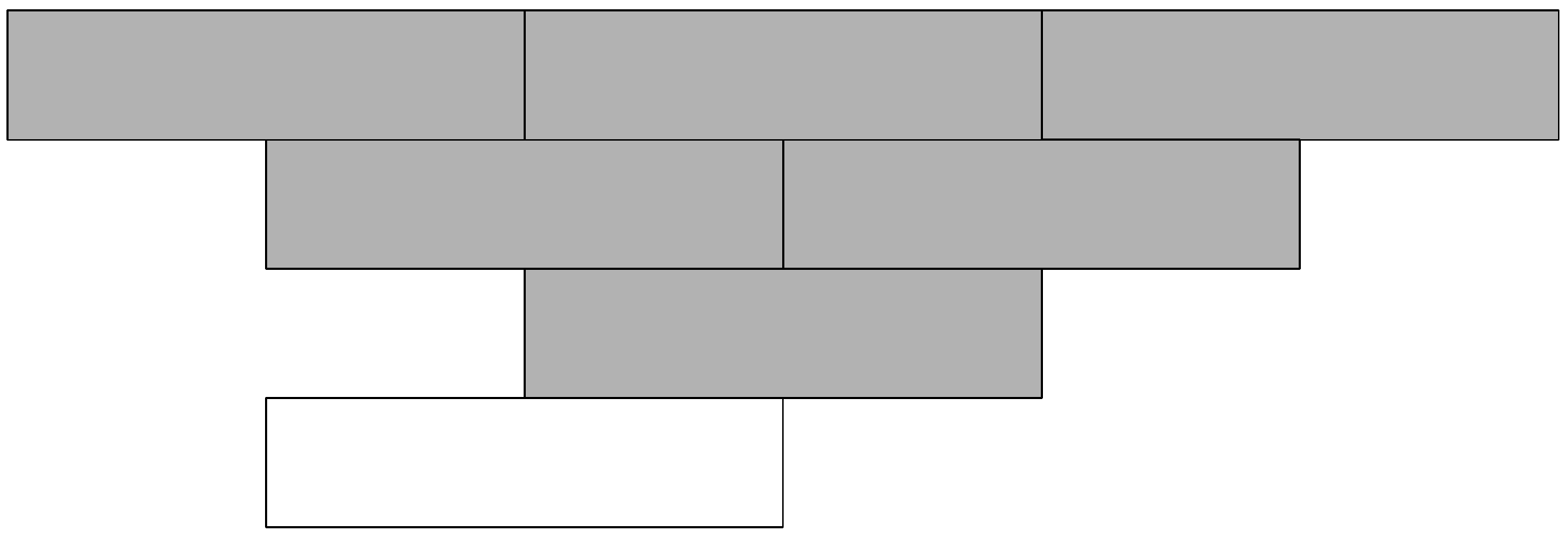}\hspace*{1cm}
\includegraphics[width=50mm]{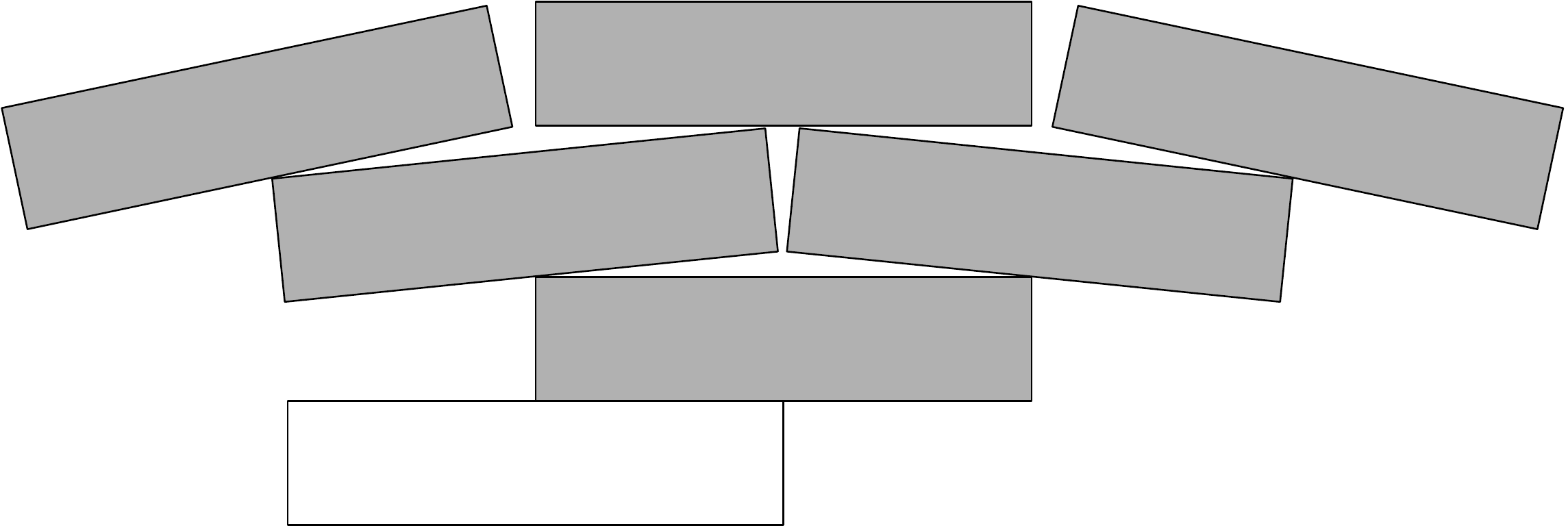}
\caption{A 3-row inverted triangle is
unbalanced.}\label{fig:pyramid}
\end{center}\vspace{-5mm}
\end{figure}


The collapse of the $3$-row triangle begins with the lifting of the
middle block in the top row. It is tempting to try to avoid this
failure by using a diamond shape instead as illustrated in
Figure~\ref{fig:collapse}. Diamonds were considered by
Drummond~\cite{D81}, and like the inverted triangle, they would
achieve an overhang of order $n^{1/2}$, though with a smaller
leading constant.
The stability analysis of diamonds is slightly more
complicated than that of inverted triangles,
but it can be shown that $d$-diamonds, i.e., diamonds that
have $d$ blocks in their largest row, are stable if and only if
$d<5$. In Figure~\ref{fig:collapse} we see a practical demonstration
with $d=5$.

\begin{figure}[h]
\begin{center}
\includegraphics[width=70mm]{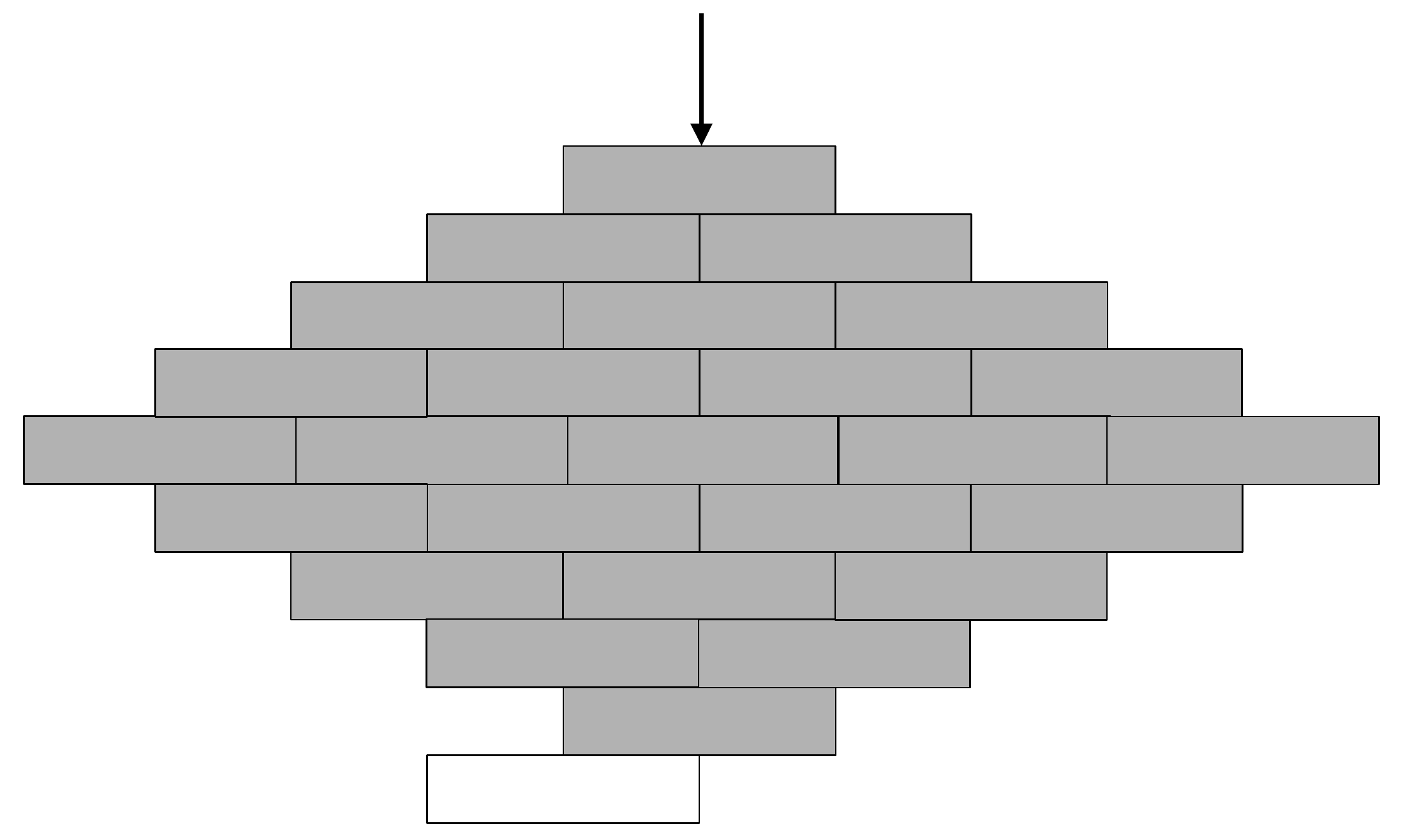}\hspace{20mm}
\includegraphics[width=70mm]{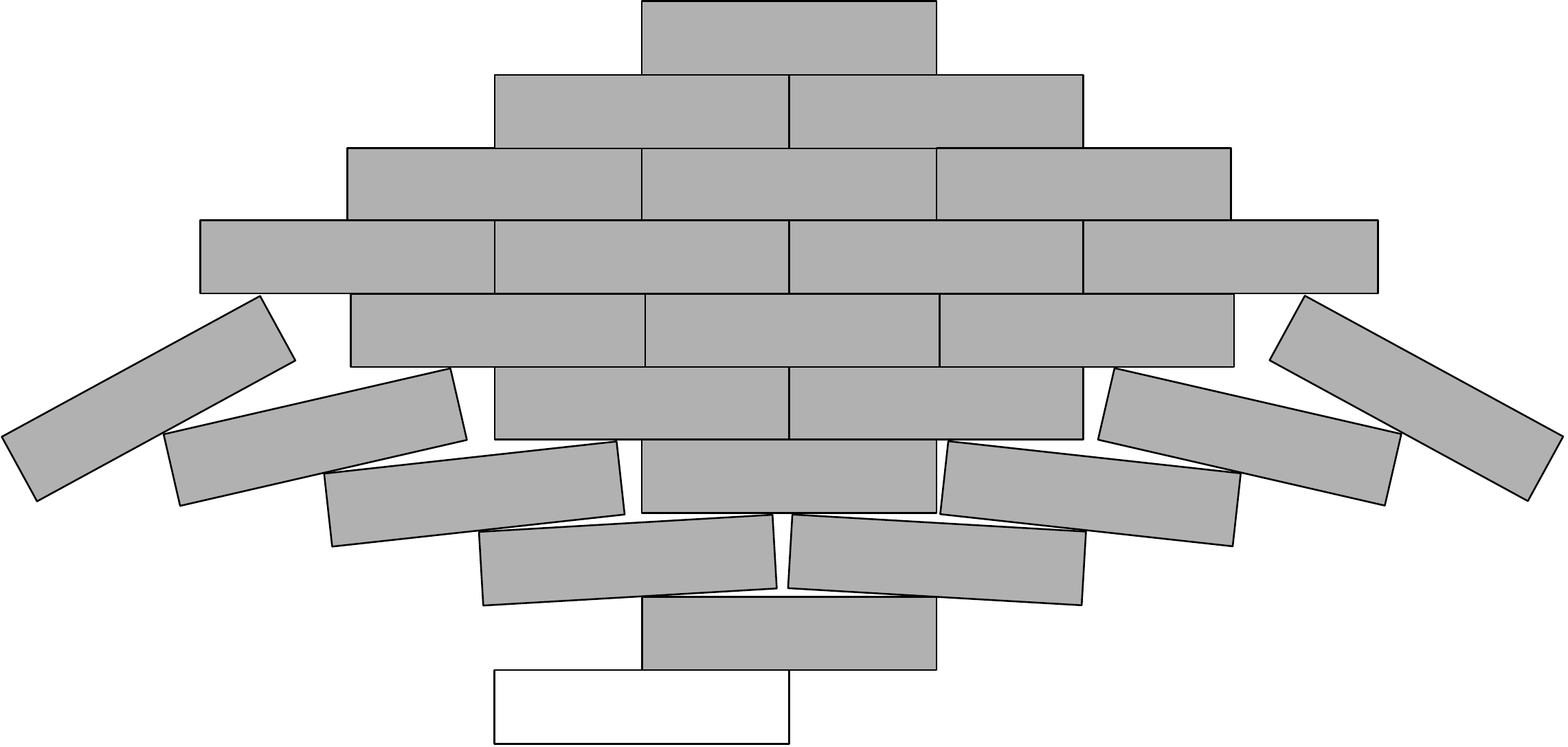}

\vspace*{0.5cm}
\includegraphics[width=70mm]{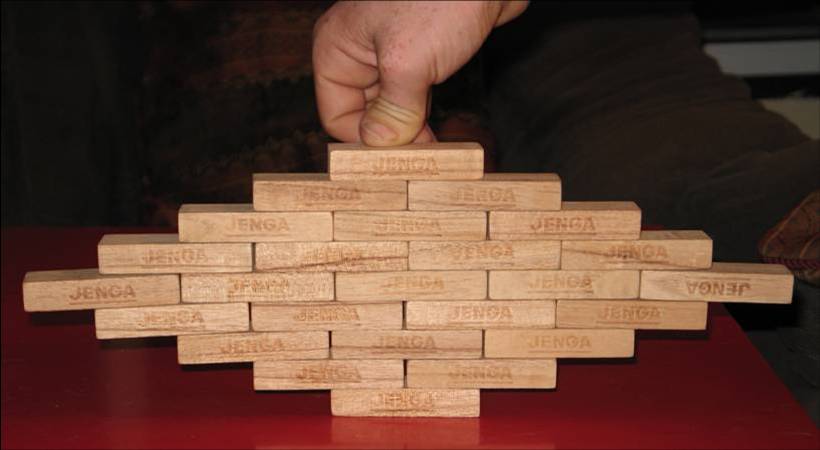}\hspace{20mm}
\includegraphics[width=70mm]{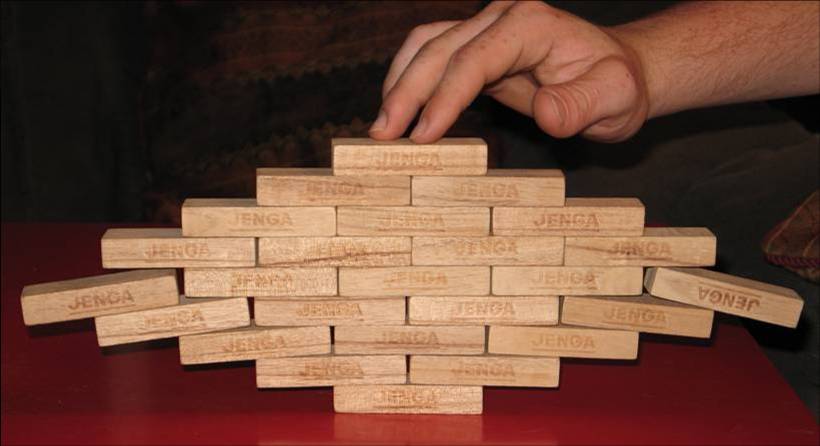}
\end{center} \caption{The instability of a $5$-diamond in theory
and practice.}\label{fig:collapse}
\end{figure}

It is not hard to show that particular constructions like larger inverted
triangles or diamonds are unstable. This instability of inverted
triangles and diamonds was
already noted in  \cite{PZ06}.
However, this does not rule out the possibility of a smarter
balanced way of stacking $n$ blocks so as to achieve an overhang
of order $n^{1/2}$, and that would be much better than the above
mentioned overhang of order $n^{1/3}$ achieved by Paterson and Zwick \cite{PZ06}. Paterson and Zwick did consider this general question.
They did not rule out an overhang
of order $n^{1/2}$, but they proved that no larger overhang
would be possible. Thus their work shows that the order of the maximum
overhang with $n$ blocks has to be somewhere between $n^{1/3}$ and $n^{1/2}$.

\subsection{Our result}
We show here that an overhang of order $n^{1/3}$, as obtained
by~\cite{PZ06}, is in fact best possible. More specifically, we show
that any $n$-block stack with an overhang of at least $6n^{1/3}$ is
unbalanced, and must hence collapse. Thus we conclude that
the maximum overhang with $n$ blocks is of order $n^{1/3}$.

\subsection{Contents}

The rest of this paper is organized as follows. In the next section
we present a precise mathematical definition of the overhang
problem, explaining in particular when a stack of blocks is said to
be \emph{balanced} (and when it is said to be \emph{stable}). In
Section~\ref{sec:wall} we briefly review the Paterson-Zwick
construction of stacks that achieve an overhang of order~$n^{1/3}$.
In Section~\ref{sec:mass} we introduce a class of abstract mass
movement problems and explain the connection between these problems
and the overhang problem. In Section~\ref{sec:bounds} we obtain
bounds for mass movement problems that imply the order $n^{1/3}$
upper bound on overhang. We end in Section~\ref{sec:concl} with some
concluding remarks and open problems.


\section{The Model} \label{sec:model}

We briefly state the mathematical definition of the overhang
problem. For more details, see \cite{PZ06}.  As in previous papers,
e.g., \cite{H05}, the overhang problem is taken here to be a two-dimensional
problem:  each block is represented by a frictionless rectangle whose long sides
are parallel to the table. Our upper bounds apply, however, in much more
general settings, as will be discussed in Section~\ref{sec:concl}.

\subsection{Stacks}

Stacks are composed of blocks that are assumed to be identical,
homogeneous, frictionless rectangles of unit length, unit weight
and height~$h$. Our results here are clearly independent of $h$,
and our figures use any convenient height.  Previous authors have
thought of blocks as cubes, books, coins, playing cards, etc.

A stack $\{B_1,\dots,B_n\}$ of $n$ blocks resting on a flat table is
specified by giving the coordinates $(x_i,y_i)$ of the lower left
corner of each block $B_i$.  We assume that the upper right corner
of the table is at $(0,0)$ and that the table extends arbitrarily
far to the left.  Thus block $B_i$ is identified with the box
$[x_i,x_i+1]\times [y_i,y_i+h]$ (its length aligned with the
$x$-axis), and the table, which we conveniently denote by $B_0$,
with the region $(-\infty,0]\times (-\infty,0]$. Two blocks are
allowed to touch each other, but their interiors must be disjoint.

We say that block $B_i$ \emph{rests} on block $B_j$, denoted
``$B_i/B_j$'', if and only if $B_i\cap B_j\ne \emptyset$ and
$y_i=y_j+h$. If $B_i\cap B_0\ne \emptyset$, then $B_i/B_0$, i.e.,
block $B_i$ rests on the table.  If $B_i/B_j$, we let
$I_{ij}=B_i\cap B_j = [a_{ij},b_{ij}]\times \{y_i\}$ be their
\emph{contact interval}. If $j\ge 1$, then $a_{ij}=\max\{x_i,x_j\}$
and $b_{ij}=\min\{x_i\!+\!1,x_j\!+\!1\}$. If $j=0$ then $a_{i0}=x_i$
and $b_{i0}=\min\{x_i\!+\!1,0\}$.

The \emph{overhang} of a stack is defined to be $\max_{i=1}^n (x_i\!+\!1)$.


\subsection{Forces, equilibrium and balance} 
\label{SS-force}


Let $\{B_1,\dots,B_n\}$ be a stack composed of~$n$ blocks. If~$B_i$
rests on~$B_j$, then~$B_j$ may apply an upward force of $f_{ij}\ge
0$ on~$B_i$, in which case~$B_i$ will reciprocate by applying a
downward force of the same magnitude on~$B_j$.  Since the blocks and
table are frictionless, all the forces acting on them are vertical.
The force~$f_{ij}$ may be assumed to be applied at a single point
$(x_{ij},y_{ij})$ in the contact interval $I_{ij}$. A downward
gravitational force of unit magnitude is applied on~$B_i$ at its
center of gravity $(x_i+\frac12,y_i+\frac{h}{2})$.



\begin{definition}[Equilibrium]\label{D:equi} Let $B$ be a homogeneous
block of unit length and unit weight, and let $a$ be the $x$-coordinate of its
left edge. Let $(x_1,f_1),(x_2,f_2),\ldots,(x_k,f_k)$ be the
positions and the magnitudes of the upward forces applied to~$B$
along its bottom edge, and let $(x'_1,f'_1), (x'_2,f'_2), \ldots,
(x'_{k'},f'_{k'})$ be the positions and magnitudes of the upward
forces applied by~$B$, along its top edge, on other blocks of the
stack. Then, $B$ is said to be in \emph{equilibrium} under these
collections of forces if and only if
$$\sum_{i=1}^k f_i \;=\; 1 + \sum_{i=1}^{k'} f'_i \ ,\quad
\sum_{i=1}^k x_i f_i \;=\; (a + \frac12) + \sum_{i=1}^{k'} x'_i
f'_i\;.$$ The first equation says that the \emph{net force} applied
to~$B$ is zero while the second says that the \emph{net moment} is
zero.
\end{definition}

\begin{definition}[Balance]\label{D:bal} A stack $\{B_1,\dots,B_n\}$ is said to be \emph{balanced}
if there exists a collection of forces acting between the blocks
along their contact intervals, such that under this collection of
forces, and the gravitational forces acting on them, all blocks are
in equilibrium.
\end{definition}


The stacks presented in Figures~\ref{fig:opt34} and~\ref{fig:PZ} are
balanced. They are, however, \emph{precariously} balanced, with some
minute displacement of their blocks leading to imbalance and
collapse. A stack can be said to be \emph{stable} if all stacks
obtained by sufficiently small displacements of its blocks are
balanced. We do not make this definition formal as it is not used in
the rest of the paper, though we refer to it in some informal
discussions.


\begin{figure}[t]
\begin{center}
\includegraphics[width=140mm]{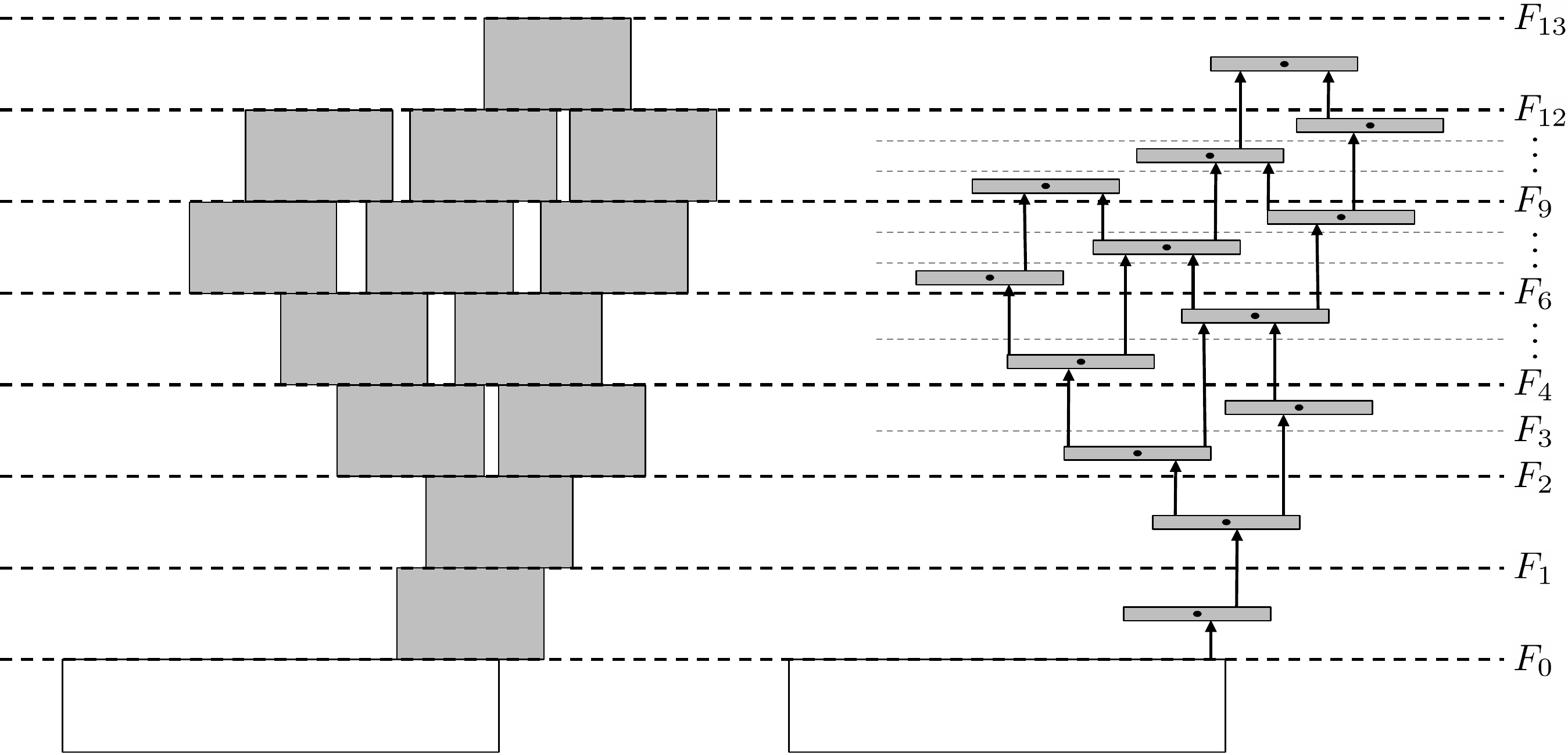}
\caption{Balancing collections of forces within a stack.}
\label{fig:FG}
\end{center}\vspace{-5mm}
\end{figure}


A schematic description of a stable stack and a collection of
balancing forces acting between its blocks is given in
Figure~\ref{fig:FG}. Only upward forces are shown in the figure but
corresponding downward forces are, of course, present. (We note in
passing that balancing forces, when they exist, are in general not
uniquely determined. This phenomenon is referred to as \emph{static
indeterminacy}.)

We usually adopt the convention that the blocks of a balanced stack
are numbered consecutively from bottom to top and from left to
right. Block $B_1$ is then the leftmost block in the lowest level
while $B_n$ is the rightmost block at the top level. For every $0\le
i\le n$, we let $F_i$ be a collection of upward balancing
forces applied by blocks in $\{B_0,B_1,\ldots,B_{i}\}$ on
blocks in $\{B_{i+1},\ldots,B_n\}$. (See Figure~\ref{fig:FG}.) We
refer to~$F_i$ as the collection of forces that cross the $i$-th
\emph{slice} of the stack.

Let us examine the relationship between two consecutive
collections~$F_i$ and~$F_{i+1}$. The only forces present in~$F_i$
but not in~$F_{i+1}$ are upward forces applied to~$B_i$, while the
only forces present in~$F_{i+1}$ but not in~$F_i$ are upward
forces applied by~$B_i$ to blocks resting upon it.  If we let
$(x_1,f_1),(x_2,f_2),\ldots,(x_k,f_k)$ be the positions and the
magnitudes of the upward forces applied to~$B_i$, and $(x'_1,f'_1),
(x'_2,f'_2), \ldots, (x'_{k'},f'_{k'})$ be the positions and
magnitudes of the upward forces applied by~$B_i$, and if we let~$a$
be the $x$-coordinate of the left edge of~$B_i$, we get by
Definitions~\ref{D:equi} and~\ref{D:bal}, that $\sum_{i=1}^k f_i = 1
+ \sum_{i=1}^{k'} f'_i$ and $\sum_{i=1}^k x_i f_i = (a + \frac12) +
\sum_{i=1}^{k'} x'_i f'_i$. Block~$B_i$ thus \emph{rearranges} the
forces in the interval $[a,a+1]$ in a way that preserves the total
magnitude of the forces and their total moment, when its own weight
is taken into account. Note that all forces of~$F_0$ act in
non-positive positions, and that if~$B_k$ is the most overhanging
block in a stack and the overhang achieved by it is~$d$, then the
total magnitude of the forces in~$F_{k-1}$ that act at or beyond
position~$d{-}1$ should be at least~$1$. These simple observations
play a central role in the rest of the paper.

\subsection{The overhang problem}

The natural formulation of the overhang problem is now:

\begin{quote}What is the maximum overhang achieved by a \emph{balanced}
$n$-block stack?
\end{quote}


The main result of this paper is:

\begin{theorem}\label{T-main} The overhang achieved by a balanced $n$-block stack
is at most $6n^{1/3}$.
\end{theorem}

The fact that the stacks in the theorem above are required to be
balanced, but not necessarily stable, makes our result only
stronger. By the nature of the overhang problem, stacks that achieve
a maximum overhang are on the verge of collapse and thus unstable.
In most cases, however, overhangs arbitrarily close to the maximum
overhang may be obtained using stable stacks. (Probably the only
counterexample is the case $n=3$.)

\section{The Paterson-Zwick construction} \label{sec:wall}

Paterson and Zwick \cite{PZ06} describe a family of balanced
$n$-block stacks that achieve an overhang of about
$(3n/16)^{1/3}\simeq 0.57n^{1/3}$. More precisely, they construct
for every integer $d\ge 1$ a balanced stack containing
$\frac{d(d-1)(2d-1)}{3}+1\simeq 2d^3/2$ blocks that achieves an
overhang of~$d/2$. Their construction, for $d=6$, is illustrated in
Figure~\ref{fig:6stack}. The construction is an example of what
\cite{PZ06} terms a \emph{brick-wall} stack, which resembles the
simple ``stretcher-bond'' pattern in real-life bricklaying. In each
row the blocks are contiguous, with each block centered over the
ends of blocks in the row beneath.
Overall the stack is symmetric and has a roughly parabolic shape,
with vertical axis at the table edge.

\begin{figure}[t]
\begin{center}
\includegraphics[width=100mm]{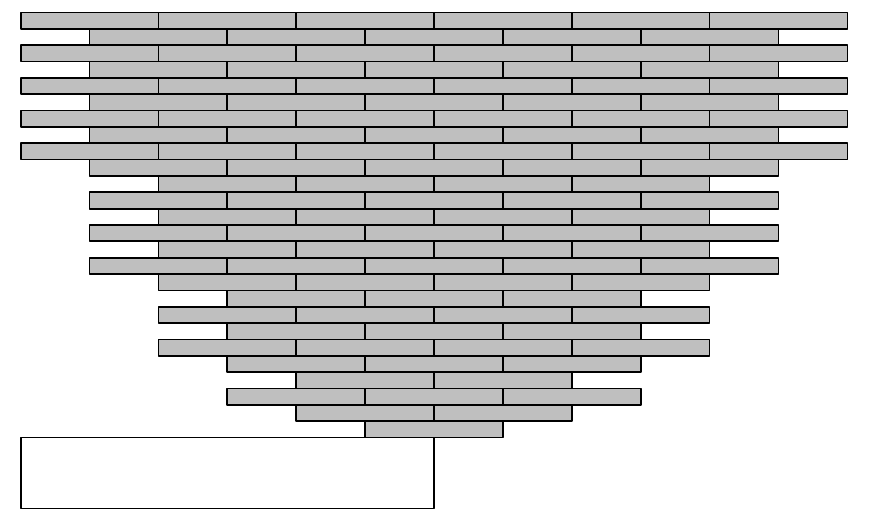}
\caption{A ``$6$-stack'' consisting of 111 blocks and giving an overhang
of $3$, taken from~\cite{PZ06}.} \label{fig:6stack}
\end{center}\vspace{-5mm}
\end{figure}

The stacks of \cite{PZ06} are constructed in the following simple
manner. A \emph{$t$-row} is a row of $t$ adjacent blocks,
symmetrically placed with respect to $x=0$. An \emph{$r$-slab} has
height $2r{-}3$ and consists of alternating $r$-rows and
$(r{-}1)$-rows, the bottom and top rows being $r$-rows. An $r$-slab
therefore contains $r(r{-}1)+(r{-}1)(r{-}2)=2(r{-}1)^2$ blocks. A
\emph{$1$-stack} is a single block balanced at the edge of the
table; a \emph{$d$-stack} is defined recursively as the result of
adding a $d$-slab symmetrically onto the top of a $(d{-}1)$-stack.
The construction itself is just a $d$-stack and so has overhang
$d/2$; its total number of blocks is given by $n = 1+ \sum_{r=1}^d
2(r{-}1)^2 = \frac{d(d-1)(2d-1)}{3}+1$. It is shown in~\cite{PZ06},
using an inductive argument, that $d$-stacks, for any $d\ge 1$, are
balanced.


Why should a parabolic shape be appropriate? Some support for this
comes from considering the effect of a block in spreading a single
force of~$f$ acting from below into two forces of almost~$f/2$ exerted
upwards from its edges. This spreading behavior is analogous to a
symmetric random walk on a line or to difference equations for
the ``heat-diffusion'' process in a linear strip.
In both cases we see that time of about~$d^2$ is
needed for effective spreading to width~$d$, corresponding to a
parabolic stack profile.

Our main result, Theorem~\ref{T-main}, states that the parabolic
stacks of \cite{PZ06} are optimal, up to constant factors. Better
constant factors can probably be obtained, however. Paterson and
Zwick \cite{PZ06} present some numerical evidence that suggests that
the overhang that can be achieved using~$n$ blocks, for large values
of~$n$, is at least $1.02n^{1/3}$. For more on this, see
Section~\ref{sec:concl}.


\section{Mass movement problems} \label{sec:mass}

Our upper bound on the maximum achievable overhang is obtained by
considering \emph{mass movement} problems that are an abstraction
of the way in which balancing forces ``flow'' though a stack of
blocks. (See the discussion at the end of Section~\ref{SS-force}.)

In a mass movement problem we are required to transform an initial
\emph{mass distribution} into a mass distribution that satisfies
certain conditions. The key condition is
that a specified amount of mass be moved to or beyond a certain
position. We can transform one mass distribution into another by
performing local \emph{moves} that redistribute mass within a given
interval in a way that preserves the total mass and the center of
mass. Our goal is then to show that many moves are required to
accomplish the task. As can be seen, masses here correspond to
forces, mass distributions correspond to collections of forces, and
moves mimic the effects of blocks.

The mass movement problems considered are formally defined in
Sections~\ref{SS-dist} and~\ref{SS-moves}. The correspondence
between the mass movement problems considered and the overhang
problem is established in Section~\ref{SS-corres}. The bounds on
mass movement problems that imply Theorem~\ref{T-main} are then
proved in Section~\ref{sec:bounds}.



\subsection{Distributions}\label{SS-dist}

\begin{definition}[Distributions and signed distributions]
A discrete mass \emph{distribution} is a set
$\mu=\{(x_1,m_1),\allowbreak (x_2,m_2),\ldots,(x_k,m_k)\}$, where~$k>0$,
$x_1,x_2,\ldots,x_k$ are real numbers, and
$m_1, \ldots, m_k> 0$. A \emph{signed distribution}~$\mu$ is
defined the same way, but without the requirement that
$m_1,m_2,\ldots,m_k> 0$.

If $\mu=\{(x_1,m_1),(x_2,m_2),\ldots,(x_k,m_k)\}$ is a (signed)
distribution, then for any set $A\subseteq \reals$, we define
$$\mu(A)=\sum_{x_i\in A} m_i\;.$$ For brevity, we use $\mu(a)$ as a
shorthand for $\mu(\{a\})$ and $\mu\{x>a\}$ as a shorthand for
$\mu(\{x\mid x>a\})$. (Note that $x$ here is a formal variable that
does not represent a specific real number.) We similarly use
$\mu\{x\ge a\}$, $\mu\{x<a\}$, $\mu\{a<x<b\}$, $\mu\{|x|\ge a\}$,
etc., with the expected meaning.

We say that a (signed) distribution is \emph{on} the interval
$[a,b]$ if $\mu(x)=0$, for every $x\not\in [a,b]$.

For every $A\subseteq \reals$, we let $\mu_A$ be the
\emph{restriction} of~$\mu$ to~$A$:
$$\mu_A \;=\; \{ (x_i,m_i) \mid x_i\in A\}\;.$$

If $\mu_1$ and $\mu_2$ are two signed distributions, we let
$\mu_1+\mu_2$ and $\mu_1-\mu_2$  be the signed distributions for
which
$$\begin{array}{c}
(\mu_1+\mu_2)(x)\;=\;\mu_1(x)+\mu_2(x),\quad \mbox{\rm for
every $x\in \reals$}\;,\\
(\mu_1-\mu_2)(x)\;=\;\mu_1(x)-\mu_2(x),\quad \mbox{\rm for
every $x\in \reals$}\;. \end{array}$$
\end{definition}

\begin{definition}[Moments]
Let $\mu=\{(x_1,m_1),(x_2,m_2),\ldots,(x_k,m_k)\}$ be a
signed distribution and let $j\ge 0$ be an integer. The $j$-th
moment of $\mu$ is defined to be:
$$\M_j[\mu] \eq \sum_{i=1}^k m_i x_i^j\;.$$
Note that~$\M_0[\mu]$ is the \emph{total mass} of~$\mu$, $\M_1[\mu]$
is the \emph{torque} of~$\mu$, with respect to the origin,
and~$\M_2[\mu]$ is the \emph{moment of inertia} of~$\mu$, again with
respect to the origin. If $\M_0[\mu] \neq 0$, we let $C[\mu] =
\M_1[\mu]/\M_0[\mu]$ be the \emph{center of mass} of~$\mu$.
\end{definition}

Less standard, but crucial for our analysis, is the following
definition.
\begin{definition}[Spread]\label{D-spread}
The \emph{spread} of a distribution
$\mu=\{(x_1,m_1),(x_2,m_2),\ldots,(x_k,m_k)\}$ is defined as
follows:
$$S[\mu]=\sum_{i<j} |x_i-x_j|\,m_i\,m_j\;.$$
\end{definition}

If $\M_0[\mu]=1$, then~$\mu$ defines a discrete random variable~$X$
for which $\Pr[X=x] = \mu(x)$, for every $x\in \reals$.  The spread
$S[\mu]$ is then half the average distance between two
independent drawings from $\mu$. We also then have $\M_1[\mu]=E[X]$
and $\M_2[\mu]=E[X^2]$. If $\M_1[\mu]=E[X]=0$, then
$\M_2[\mu]=E[X^2]=\mbox{\it Var}[X]$. It is also worthwhile noting
that if $\mu_1$ and $\mu_2$ are two distributions then, for any
$k\ge 0$, $\M_k[\mu_1+\mu_2]=\M_k[\mu_1]+\M_k[\mu_2]$, i.e., $\M_k$
is a linear operator.

An inequality that proves very useful in the sequel is the
following:

\begin{lemma}\label{L-SM2} For any discrete distribution $\mu$ we have
$S[\mu]^2\leq \frac{1}{3}\M_2[\mu]\M_0[\mu]^3$.
\end{lemma}

The proof of Lemma~\ref{L-SM2} is given in Section~\ref{SS-proofs}.

\subsection{Mass redistribution moves}\label{SS-moves}


\begin{definition}[Moves]\label{D-move}
A \emph{move} $v=([a,b],\delta)$ consists of an interval $[a,b]$
and a signed distribution $\delta$ on~$[a,b]$ with
$\M_0[\delta]=\M_1[\delta]=0$. A move~$v$ can be \emph{applied} to
a distribution $\mu$ if the signed distribution $\mu'=\mu+\delta$
is a distribution, in which case we denote the result~$\mu'$ of
this application by $v\mu$. We refer to $\frac{a+b}{2}$ as the
\emph{center} of the move. Unless otherwise stated, the moves we
consider operate on intervals of length~$1$, i.e., $b-a=1$.
\end{definition}

Note that $v$ is a move and $\mu'=v\mu$, then
$M_0[\mu']=M_0[\mu]$, $M_1[\mu']=M_1[\mu]$ and consequently
$C[\mu']=C[\mu]$.

A sequence $V=\langle
v_1,v_2,\ldots,v_\ell\rangle$ of moves and an initial distribution
$\mu_0$ naturally define a sequence of distributions
$\mu_0,\mu_1,\ldots,\mu_\ell$, where $\mu_i=v_i\mu_{i-1}$ for
$1\le i\le \ell$. (It is assumed here that $v_i$ can indeed be
applied to $\mu_{i-1}$.) We let $V\mu_0=\mu_\ell$.

Moves and sequences of moves simulate the behavior of
\emph{weightless} blocks and stacks. However, the blocks that we are
interested in have unit weight.
Instead of explicitly taking into account the weight of the blocks,
as we briefly do in Section~\ref{SS-corres},
it turns out that it is enough for our purposes to impose a natural
restriction on the move sequences considered.
We start with the following definition:


\begin{definition}[$\mu_{\max}$] If $\mu_0,\mu_1,\ldots,\mu_\ell$ is a sequence of
distributions, and $a\in \reals$, we define
$$
\mu_{\max}\{x> a\} \;=\; \max_{0\le i\le \ell} \mu_i\{x> a\}\;.
$$
Expressions like $\mm{x\ge a}$, $\mm{x<a}$ and $\mm{x\le a}$ are
defined similarly.
\end{definition}

\begin{definition}[Weight-constrained sequences]
A sequence $V=\langle v_1,v_2,\ldots,v_\ell\rangle$ of moves that
generates a sequence $\mu_0,\mu_1,\ldots,\mu_\ell$ of distributions
is said to be \emph{weight-constrained}, with respect to~$\mu_0$ if,
for every $a\in \reals$, the number of moves in~$V$ centered
in~$(a,\infty)$ is at most $\mm{x>a}$.
\end{definition}

The two main technical results of this paper are the following theorems.

\begin{theorem}\label{T-m1}
If a distribution~$\nu$ is obtained from a distribution~$\mu$ with
$\mu\{x\le 0\}\le n$ and $\mu\{x>0\}=0$, where $n\geq 1$, by a
weight-constrained move sequence, then $\nu\{x\ge 6 n^{1/3}-1\}=0$.
\end{theorem}

For general move sequences we have the following almost tight
result, which might be of some independent interest.
In particular,
it shows that the weight constraint only has a logarithmic
effect on the maximal overhang.

\begin{theorem}\label{T-m2}
If a distribution~$\nu$ is obtained from a distribution~$\mu$ with
$\mu\{x\le 0\}\le n$ and $\mu\{x>0\}=0$, where $n\geq 1$, by a move
sequence of length at most $n$, then $\nu\{x\ge 2 n^{1/3}\log_2
n\}<1$.
\end{theorem}

We show next that Theorem~\ref{T-m1} does indeed imply
Theorem~\ref{T-main}, the main result of this paper.

\subsection{From overhang to mass movement}\label{SS-corres}

The moves of Definition~\ref{D-move} capture the essential effect
that a block can have on the collections of forces within a stack.
They fail to take into account, however, the fact that the
\emph{weight} of a block is ``used up'' by the move and is then
lost. To faithfully simulate the effect of unit weight blocks we
introduce the slightly modified definition of \emph{lossy moves}:

\begin{definition}[Lossy moves]\label{D-lossy}
%
If $v=([a,b],\delta)$ is a move, then the \emph{lossy} move
$v^\downarrow$ associated with it is
$v^\downarrow=([a,b],\delta^\downarrow)$, where $\delta^\downarrow =
\delta-\{(\frac{a+b}{2},1)\}$. A lossy move $v^\downarrow$ can be
applied to a distribution $\mu$ if $\mu'=\mu+\delta^\downarrow$ is a
distribution, in which case we denote the result~$\mu'$ of this
application by $v^\downarrow\mu$.
\end{definition}

Note that if $v^\downarrow=([a,b],\delta^\downarrow)$ is a lossy
move and $\mu'=v^\downarrow \mu$, then $M_0[\mu']=M_0[\mu]-1$ and
$M_1[\mu']=M_1[\mu]-\frac{a+b}{2}$. Hence, lossy moves
do not preserve total mass or center of mass.

If $V=\langle v_1,v_2,\ldots,v_\ell\rangle$ is a sequence of moves,
we let $V^\downarrow=\langle v^\downarrow_1, v^\downarrow_2, \ldots,
v^\downarrow_\ell\rangle$ be the corresponding sequence of lossy
moves. If~$\mu_0$ is an initial distribution, we can naturally
define the sequence of distributions $\mu_0,\mu_1, \ldots,
\mu_\ell$, where $\mu_i=v^\downarrow_i \mu_{i-1}$  for $1\le i\le
\ell$, obtained by applying~$V^\downarrow$ to~$\mu_0$.

A collection of forces~$F_i$ may also be viewed as mass
distribution. The following lemma is now a simple formulation of the
definitions and the discussion of Section~\ref{SS-force}:

\begin{lemma}\label{L-block} Let $\{B_1,B_2,\ldots,B_n\}$ be a balanced stack. Let $F_i$ be
a collection of balancing forces acting between $\{B_0,\ldots,B_i\}$
and $\{B_{i+1},\ldots,B_n\}$, for $0\le i\le n$. Let~$x_i$ be the
$x$-coordinate of the left edge of~$B_i$. Then, $F_{i+1}$ can be
obtained from~$F_i$ by a lossy move in the interval $[x_i,x_i{+}1]$.
\end{lemma}

As an immediate corollary, we get:

\begin{lemma}\label{L-c1} If there is a stack composed of~$n$
blocks of length~$1$ and weight~$1$ that achieves an overhang
of~$d$, then there is sequence of at most~$n{-}1$ lossy moves that
transforms a distribution~$\mu$ with $\M_0[\mu]= \mu\{x\le 0\}=n$
and $\mu\{x>0\}=0$ into a distribution~$\mu'$ with $\mu'\{x\ge d{-}1\}
\ge 1$.
\end{lemma}

\begin{proof} Let $\{B_1,B_2,\ldots,B_n\}$ be a balanced stack and
let~$B_k$ be a block in it that achieves an overhang of~$d$. As
before, we let~$F_i$ be a collection of balancing forces acting
between $\{B_0,\ldots,B_i\}$ and $\{B_{i+1},\ldots,B_n\}$. We let
$\mu=F_0$ and $\mu'=F_{k{-}1}$. It follows from Lemma~\ref{L-block}
that $\mu'$ may be obtained from~$\mu$ by a sequence of~$k{-}1$
lossy moves. As all the forces in $\mu=F_0$ are forces applied by
the table $B_0$, and as the table supports the weight of the~$n$
blocks of the stack, we have $\M_0[\mu_0]= \mu_0\{x\le 0\} =n$ and
$\mu\{x>0\}=0$. As the forces in $\mu'=F_{k-1}$ must at least
support the weight of~$B_k$, we have $\mu'\{d{-}1\le x\le d\}\ge 1$.
\end{proof}

The next simple lemma shows that sequences of lossy moves can be
easily converted into weight-constrained sequences of moves and
distributions that ``dominate'' the original sequence.

\begin{lemma} If $\mu_0,\mu_1,\ldots,\mu_\ell$ is a sequence of
distributions obtained by a sequence of \emph{lossy} moves, then
there exists a sequence of distributions
$\mu'_0,\mu'_1,\ldots,\mu'_\ell$ obtained by a \emph{weight-constrained}
sequence of moves such that $\mu'_0=\mu_0$, and
$\mu'_i(x)\ge \mu_i(x)$, for every $1\le i\le \ell$ and $x\in
\reals$.
\end{lemma}

\begin{proof} The sequence $\mu'_0,\mu'_1,\ldots,\mu'_\ell$ is
obtained by performing exactly the same moves used to obtain the
sequence $\mu_0,\mu_1,\ldots,\mu_\ell$, treating them now as moves
rather than lossy moves. More formally, if
$\mu_i=v_i^\downarrow\mu_{i-1}$, we let $\mu'_i=v_i\mu'_{i-1}$. If
$v_i=([a-\frac12,a+\frac12],\delta)$, then $\mu'_i$ now has an extra
mass of size~$1$ at $a$. This mass is \emph{frozen}, and will not be
touched by subsequent moves. Hence, if~$k$ moves have their center
beyond position~$a$, then $\mu'_{{\max}}\{x>a\}\ge
\mu'_\ell\{x>a\}\ge k$, as required by the definition of
weight-constrained sequences.
\end{proof}

It is now easy to see that Theorem~\ref{T-m1} together with
Lemmas~\ref{L-block} and~\ref{L-c1} imply Theorem~\ref{T-main}.

\section{Bounds on mass movement problems}\label{sec:bounds}

This section is devoted to the proofs of Theorems~\ref{T-m1}
and~\ref{T-m2}. As mentioned, Theorem~\ref{T-m1} implies
Theorem~\ref{T-main}, which states that an $n$-block stack can have
an overhang of at most $6n^{1/3}$.

\subsection{Extreme moves and splits}

We begin by considering an important class of moves:

\begin{definition}[Extreme moves]\label{D-extreme}
An \emph{extreme move} $\bar{v}$ is defined solely as an interval
$[a,b]$. An extreme move $\bar{v}$ can be applied to any
distribution~$\mu$ resulting in the distribution $\mu'=\bar{v}\mu$
such that $\mu'\{a{<}x{<}b\}=0$, $\mu'(x)=\mu(x)$ for every
$x\not\in [a,b]$, $\M_0[\mu]=\M_0[\mu']$ and $\M_1[\mu]=\M_1[\mu']$.
In other words, an extreme move moves all the mass in the interval
$[a,b]$ into the endpoints of this interval while maintaining the
center of mass.
If~$v$ is a move on an interval $[a,b]$, we let~$\bar{v}$ denote
the extreme move on $[a,b]$. If~$V$ is a sequence of moves, we
let~$\bar{V}$ denote the corresponding sequence of extreme moves.
\end{definition}

Closely related to Lemma~\ref{L-SM2} is the following lemma:

\begin{lemma}\label{L-extreme} If $\mu_1$ is obtained from $\mu_0$ by an
extreme move (in an interval of length~1) then
$$S[\mu_1]-S[\mu_0] \;\ge\; 3(\M_2[\mu_1]-\M_2[\mu_0])^2\;.$$
\end{lemma}

The proof of Lemma~\ref{L-extreme} is again deferred to
Section~\ref{SS-proofs}.

We next define a natural partial order on distributions:

\begin{definition}[Splitting]
Let $\mu$ and $\mu'$ be two distributions. We say that $\mu'$
\emph{is a basic split of}~$\mu$, denoted $\mu\preceq_1\mu'$, if
$\mu'$ is obtained by taking one of the point masses $(x_i,m_i)$
of~$\mu$ and replacing it by a collection
$\{(x'_1,m'_1),\ldots,(x'_\ell,m'_\ell)\}$ of point masses with
total mass $m_i$ and center of mass at~$x_i$. We say that~$\mu'$
\emph{splits into}~$\mu$, denoted $\mu\preceq \mu'$, if~$\mu'$ can
be obtained from~$\mu$ by a sequence of zero or more basic splits.
\end{definition}

The following two lemmas summarize simple properties of splits and
extreme moves that will be explicitly or implicitly used in this
section. Their obvious proofs are omitted.

\begin{lemma}\label{L-simple1}\ \vspace{-3mm}
\begin{enumerate}
\item[(i)] If $\mu\preceq \mu'$ and $\mu'\preceq \mu''$, then
$\mu\preceq \mu''$.
\item[(ii)] If $\mu_1\preceq \mu_1'$ and $\mu_2\preceq \mu_2'$, then
$\mu_1+\mu_2 \preceq \mu'_1+\mu'_2$.
\item[(iii)] For any distribution~$\mu$ we have $\{(\C[\mu],M_0[\mu[)\}
\preceq \mu$.
\item[(iv)] If $\mu=\{(x_1,m_1),(x_2,m_2)\}$ and
$\mu'=\{(x'_1,m'_1),(x'_2,m'_2)\}$, where $x'_1\le x_1\le x_2\le x'_2$,
$\M_0[\mu]=\M_0[\mu']$ and $\C[\mu]=\C[\mu']$, then $\mu\preceq
\mu'$.
\end{enumerate}
\end{lemma}

\begin{lemma}\label{L-simple2}\ \vspace{-3mm}
\begin{enumerate}
\item[(i)] If $v\mu$ is defined then $v\mu\preceq \bar{v}\mu$.
\item[(ii)] If $\bar{v}$ is an extreme move then $\mu \preceq
\bar{v}\mu$.
\item[(iii)] If $\bar{v}$ is an extreme move then $\bar{v}(\mu_1+\mu_2) =
\bar{v}\mu_1+\bar{v}\mu_2$.
\end{enumerate}
\end{lemma}

The following lemma shows that splitting increases the second
moment.

\begin{lemma}\label{L-mom} If $\mu\preceq \mu'$ then $\M_2[\mu]\le \M_2[\mu']$.
\end{lemma}

\begin{proof} Due to the linearity of~$\M_2$ and the fact that~$\preceq$
is the transitive closure of~$\preceq_1$, it is enough to prove the
claim when $\mu=\{(x,m)\}$ is composed of a single mass and
$\mu'=\{(x'_1,m'_1),\ldots,(x'_k,m'_k)\}$ is obtained from~$\mu$ by
a basic split. For any
distribution~$\nu=\{(x_1,m_1),\ldots,(x_k,m_k)\}$ and any
$c\in\reals$ we define $\M_2[\nu,c]=\sum_{i=1}^k m_i(x_i-c)^2$ to be
the second moment of~$\nu$ about~$c$. As $\M_0[\mu]=\M_0[\mu']$ and
$\M_1[\mu]=\M_1[\mu']$, a simple calculation shows that
$\M_2[\mu',c]-\M_2[\mu,c]=\M_2[\mu']-\M_2[\mu]$, for any
$c\in\reals$. Choosing $c=x$ and noting that $\M_2[\mu,x]=0$ while
$\M_2[\mu',x]\ge 0$, we get the required inequality.
\end{proof}

The next lemma exhibits a relation between extreme moves and splitting.

\begin{lemma}\label{L-splitmove}
If $\mu\preceq \mu'$ and $v$ is a move that can be applied to~$\mu$,
then $v\mu \preceq \bar{v}\mu'$.
\end{lemma}

\begin{proof} We show that $v\mu \preceq \bar{v}\mu \preceq
\bar{v}\mu'$, and use~Lemma~\ref{L-simple1}$(i)$. The first relation
is just Lemma~\ref{L-simple2}$(i)$.
%
It remains to show $\bar{v}\mu \preceq \bar{v}\mu'$.
By Lemma~\ref{L-simple2}$(iii)$, it is enough to prove the claim for
$\mu=\{(x,m)\}$ composed of a single mass.
Let $[a,b]$ be the interval corresponding to~$\bar{v}$. There are
two cases. If $x\not\in [a,b]$, then  $$\bar{v}\mu \;=\; \mu
\;\preceq\; \mu' \;\preceq\; \bar{v}\mu'\;,$$ as required.
The more interesting case is when $x\in [a,b]$. Let
$\nu=\bar{v}\mu=\{(a,m_1),(b,m_2)\}$ and $\nu'=\bar{v}\mu'$. Let
$\nu'_\ell = \mu_{(-\infty,a]}$ and $\nu'_r = \mu_{[b,\infty)}$. As
$\bar{v}$ leaves no mass in $(a,b)$, we get that
$\nu'=\nu'_\ell+\mu'_r$. Let $\bar{m}_\ell=\M_0[\nu'_\ell]$,
$\bar{m}_r=\M_0[\nu'_r]$, $\bar{x}_\ell=\C[\nu'_\ell]$ and
$\bar{x}_r=\C[\nu'_r]$. As $\bar{x}_\ell\le a<b\le \bar{x}_r$, we
get using Lemma~\ref{L-simple1}$(ii)$ and $(iii)$ that
$$\nu \eq \{(a,m_1),(b,m_2)\} \;\preceq\;
\{(\bar{x}_\ell,\bar{m}_\ell),(\bar{x}_r,\bar{m}_r)\} \eq
\{(\bar{x}_\ell,\bar{m}_\ell)\} + \{(\bar{x}_r,\bar{m}_r)\}
\;\preceq\; \nu'_\ell+\nu'_r \eq \nu'\;,$$ as required.
%
%
\end{proof}

Using induction we easily obtain:

\begin{theorem}\label{T-splits} If $V$ is a sequence of moves that can be
applied to~$\mu$, then $V \mu\preceq \bar{V}\mu$.
\end{theorem}

Combining Theorem~\ref{T-splits} and Lemma~\ref{L-mom} we get the
following immediate corollary.

\begin{corollary}\label{C-extreme} If $V$ is a sequence of moves that can be
applied to~$\mu$, then $\M_2[V\mu] \le \M_2[\bar{V}\mu]$.
\end{corollary}

\subsection{Spread vs.\ second moment}

We now obtain our first bound for mass movement problems. The bound
relies heavily on Lemma~\ref{L-SM2} that relates the spread and
second moment of a distribution, on Lemma~\ref{L-extreme} that
relates differences in spread to differences in second moments,
and finally, on Corollary~\ref{C-extreme} that states that
converting moves to extreme moves can only increase the second
moment.

\begin{lemma}\label{L-initial} Any sequence of moves that transforms the distribution
$\mu=\{(0,1)\}$ into a distribution $\nu$ with $\nu\{|x|\ge \NN\}\ge
p$, where $\NN>0$ and $0<p<1$, must contain at least
$(3p)^{3/2}\NN^3$ moves.
\end{lemma}

\begin{proof}
Let $\mu_0,\mu_1,\ldots,\mu_\ell$ be the sequence of
distributions obtained by applying a sequence~$V$ of~$\ell$ moves to
$\mu_0=\{(0,1)\}$, and suppose that $\mu_\ell\{|x|\ge \NN\}\ge p$.
By the definition of the second moment we have $\M_2[\mu_\ell]\ge
p\NN^2$.

Let $\bar{\mu}_0,\bar{\mu}_1,\ldots,\bar{\mu}_\ell$ be
the sequence of distributions obtained by applying the sequence
$\bar{V}$ of the extreme moves corresponding to the moves of~$V$
on $\bar{\mu}_0=\mu_0=\{(0,1)\}$.
By Corollary~\ref{C-extreme}, we get that
$$\M_2[\bar{\mu}_\ell]\;\ge\;\M_2[\mu_\ell]\;\ge\; p\,\NN^2\;.$$
By Lemma~\ref{L-SM2} we have
$$\frac{\M_2[\bar{\mu}_\ell]^2}{S[\bar{\mu}_\ell]}
 \;=\; \left(\frac{\M_2[\bar{\mu}_\ell] \M_0[\bar{\mu}_\ell]^3}{S[\bar{\mu}_\ell]^2}\right)^{1/2}
  \M_2[\bar{\mu}_\ell]^{3/2}
 \;\ge\; {\sqrt{3}}\,\M_2[\bar{\mu}_\ell]^{3/2} \;\ge\; {\sqrt{3}}p^{3/2}\NN^3\;.$$

Let~$h_i=\M_2[\bar{\mu}_i]-\M_2[\bar{\mu}_{i-1}]$, for $1\le i\le
\ell$. As $\M_2[\bar{\mu}_0]=0$, we clearly have,
$$\M_2[\bar{\mu}_\ell]\;=\;\sum_{i=1}^\ell h_i\;.$$
By Lemma~\ref{L-extreme}, we get that
$$S[\bar{\mu}_\ell]\;\ge\; {3}\sum_{i=1}^\ell h_i^2\;.$$
Using the Cauchy-Schwartz inequality to justify the second
inequality below, we get:
$$
S[\bar{\mu}_\ell] \;\ge\; {3}\sum_{i=1}^\ell h_i^2 \; \ge\;
{3}\frac{(\sum_{i=1}^\ell h_i)^2}{\ell} \;=\;
{3}\frac{\M_2[\bar{\mu}_\ell]^2}{\ell}\;.$$
Thus, as claimed,
$$\ell\;\ge\; 3\frac{\M_2[\bar{\mu}_\ell]^2}{S[\bar{\mu}_\ell]} \;\ge\; (3p)^{3/2}\NN^3\;.
\vspace{-2ex}$$
\end{proof}

\subsection{Mirroring}
The main result of this section is:

\begin{theorem}\label{T-asym}
Let $\mu_0,\mu_1,\ldots,\mu_\ell$ be a sequence of distributions
obtained by applying a sequence of moves to an initial distribution
$\mu_0$ with $\mu_0\{x>r\}=0$. If $\mm{x>r}\le m$ and $\mm{x\ge
r+d}\ge pm$, where $d>1$ and $0<p<1$, then the sequence of
moves must contain at least $\sqrt{3}p^{3/2}(d-\frac{1}{2})^3$ moves
whose centers are in $(r+\frac{1}{2},\infty)$.
\end{theorem}

The theorem follows immediately from the following lemma by shifting
coordinates and renormalizing masses.

\begin{lemma}\label{L-max-g}
Let $\mu_0,\mu_1,\ldots,\mu_\ell$ be a sequence of distributions
obtained by applying a sequence of moves to an initial distribution
$\mu_0$ with $\mu_0\{x>-\frac{1}{2}\}=0$. If $\MM{x>-\frac{1}{2}}\le
1$ and $\mm{x\ge \NN}\ge p$, where $\NN>\frac{1}{2}$ and $0<p<1$,
then the sequence of moves must contain at least
$\sqrt{3}p^{3/2}\NN^3$ moves whose centers are at strictly positive
positions.
\end{lemma}

\begin{proof} We may assume, without loss of generality, that the
first move in the sequence moves some mass from $(-\infty,\frac12]$
into $(\frac{1}{2},\infty)$ and that the last move moves some mass
from $(-\infty,d)$ to $[d,\infty)$. Hence, the center of the first
move must be in $(-1,0]$ and the center of the last move must be at
a positive position.

We shall show how to transform the sequence of distributions
$\mu_0,\mu_1,\ldots,\mu_\ell$ into a sequence of distributions
$\mu'_0,\mu'_1,\ldots,\mu'_{\ell'}$, obtained by applying a sequence
of~$\ell'$ moves, such that $\mu'_0=\{(0,1)\}$,
$\mu'_{\ell'}\{|x|\ge d\}\ge p$, and such that the number of
moves~$\ell'$ in the new sequence is at most three times the
number~$\ell^+$ of positively centered move in the original
sequence. The claim of the lemma would then follow immediately from
Lemma~\ref{L-initial}.

The first transformation is ``negative truncation'', where in each
distribution $\mu_i$, we shift mass from the interval
$(-\infty,-\frac{1}{2})$ to the point $-\frac{1}{2}$. Formally the
resulting distribution $\fvec{\mu}_i$ is defined by
\[\fvec{\mu}_i(x)\; =\; \left\{\begin{array}{ll}
\mu_i(x)&\mbox{ if }x>-\frac{1}{2}\\
1-\mu_i\{x> -\frac{1}{2}\}&\mbox{ if }x=-\frac{1}{2}\\
0&\mbox{ if }x<-\frac{1}{2}
\end{array}\right.  .
\]
Note that the total mass of each distribution is $1$ and that
$\fvec{\mu}_0=\{(-\frac{1}{2},1)\}$. Let $\delta_i=\mu_i-\mu_{i-1}$
be the signed distribution associated with the move that transforms
$\mu_{i-1}$ into $\mu_i$ and let $[c_i-\frac12,c_i+\frac12]$ be the
interval in which it operates. For brevity, we refer to $\delta_i$
as the move itself, with its center $c_i$ clear from the context. We
now compare the transformed ``moves''
$\fvec{\delta}_i=\fvec{\mu}_i-\fvec{\mu}_{i-1}$ with the original
moves $\delta_i=\mu_i-\mu_{i-1}$. If $c_i>0$, then $\delta_i$ acts
above $-\frac{1}{2}$ and $\fvec{\delta}_i=\delta_i$. If $c_i\le -1$,
then $\delta_i$ acts at or below $-\frac{1}{2}$, so
$\fvec{\delta}_i$ is null and $\fvec{\mu}_i=\fvec{\mu}_{i-1}$. In
the transformed sequence, we skip all such null moves. The remaining
case is when the center~$c_i$ of~$\delta_i$ is in $(-1,0]$. In this
case~$\fvec{\delta}_i$ acts within $[-\frac{1}{2},\frac{1}{2}]$, and
we view it as centered at~$0$. However, typically~$\fvec{\delta}_i$
does not define a valid move as it may change the center of gravity.
We call these~$\fvec{\delta}_i$ \emph{semi-moves}. If we have two
consecutive semi-moves $\fvec{\delta}_i$ and $\fvec{\delta}_{i+1}$,
we combine them into a single semi-move
$\fvec{\delta}_i+\fvec{\delta}_{i+1}$, taking~$\fvec{\mu}_{i-1}$
directly to~$\fvec{\mu}_{i+1}$. In the resulting negatively
truncated and simplified sequence, we know that at least every
alternate move is an original, positively centered, move. Since the
last move in the original sequence was positively centered we
conclude:
\begin{claim}\label{C:neg-truncation}
The sequence obtained by the negative truncation transformation and
the subsequent clean-up is composed of original positively centered
moves and semi-moves (acting within $[-\frac{1}{2},\frac{1}{2}]$).
The sequence begins with a semi-move and at most half of its
elements are semi-moves.
\end{claim}
Next, we create a reflected copy of the negatively truncated
distributions. The reflected copy $\bvec{\mu}_i$ of $\fvec{\mu}_i$
is defined by
\[ \bvec{\mu}_i(x)\;=\;\fvec{\mu}_i(-x),\quad \mbox{for every $x\in\reals$}.\]
We similarly define the reflected (semi-)moves
$\bvec{\delta}_i=\bvec{\mu}_i-\bvec{\mu}_{i-1}$. We can now define
the mirrored distributions
$$\bfvec{\mu}_i \eq \fvec{\mu}_i+\bvec{\mu}_i\;.$$
Note that
$\bfvec{\mu}_0=\fvec{\mu}_0+\bvec{\mu}_0=\{(-\frac{1}{2},1),(\frac{1}{2},1)\}$.
The distribution $\bfvec{\mu}_i$ may be obtained from
$\bfvec{\mu}_{i-1}$ by applying the (semi-)move $\fvec{\delta}_i$,
resulting in the distribution $\fvec{\mu}_i+\bvec{\mu}_{i-1}$, and
then the (semi-)move $\bvec{\delta}_i$, resulting in
$\fvec{\mu}_i+\bvec{\mu}_{i}=\bfvec{\mu}_i$. The $\bfvec{\mu}_i$
sequence is therefore obtained by interleaving the (semi-)moves
$\fvec{\delta}_i$ with their reflections $\bvec{\delta}_i$. Now
comes a key observation:
\begin{claim} If $\fvec{\delta}_i$ and $\bvec{\delta}_i$ are semi-moves, then
their sum $\bfvec{\delta}_i=\fvec{\delta}_i+\bvec{\delta}_i$ defines
an ordinary move centered at~$0$ and acting on
$[-\frac{1}{2},\frac{1}{2}]$.
\end{claim}
\begin{proof} Both $\fvec{\delta}_i$ and $\bvec{\delta}_i$ preserve the total mass.
As $\bfvec{\delta}_i$ is symmetric about $0$, it cannot change the
center of mass.
\end{proof}

As suggested by the above observation, if $\fvec{\delta}_i$ and
$\bvec{\delta}_i$ are semi-moves, we combine them into a single
ordinary move $\bfvec{\delta}_i$ centered at~$0$.
%
%
%
%
We thus obtain a sequence of at most $3\ell^+$ moves, where $\ell^+$
is the number of positively centered moves in the original sequence,
that transforms $\bfvec{\mu}_0$ to $\bfvec{\mu}_\ell$.

Recall from Claim~\ref{C:neg-truncation} that the first ``move''
$\fvec{\delta}_1$ in the negatively truncated sequence is a
semi-move. The first move $\bfvec{\delta}_1$, obtained by combining
$\fvec{\delta}_1$ and $\bvec{\delta}_1$, is therefore a move acting
on $[-\frac12,\frac12]$. We now replace the initial distribution
$\bfvec{\mu}_0=\{(-\frac{1}{2},1),(\frac{1}{2},1)\}$ by the
distribution $\mu'_0=\{(0,2)\}$, which has the same center of
gravity, and replace the first move by
$\delta'_1=\bfvec{\delta}_1+\{(-\frac{1}{2},1),(0,-2),(\frac{1}{2},1)\}$.
The distribution after the first move is then again $\bfvec{\mu}_1$.

We have thus obtained a sequence of at most $3\ell^+$ moves
that transforms $\mu'_0=\{(0,2)\}$ into a distribution
$\nu'=\bfvec{\mu}_\ell$ with $\nu'\{|x|\ge \NN\}\ge 2p$. Scaling
these distribution and moves by a factor of~2, we get, by
Lemma~\ref{L-initial}, that $3\ell^+\ge (3p)^{3/2}\NN^3$, as
claimed.
\end{proof}

\subsection{Proofs of Theorems~\ref{T-m1} and~\ref{T-m2}}

We prove the following theorem which easily implies
Theorem~\ref{T-m1}.

\begin{theorem}\label{T-gen} Let $\mu_0,\mu_1,\ldots,\mu_\ell$ be a sequence of
distributions obtained by applying a constrained sequence of moves
on an initial distribution $\mu_0$ with $\mu_0\{x>r \}=0$. If
$\mm{x>r}\leq n$, where $n\geq \frac15$, then $\MM{x> r + 6
n^{1/3}-1}=0$.
\end{theorem}


\begin{proof}
The proof is by induction on~$n$. If $n<1$ then there is no move
with center greater than $r$ and hence $\MM{x> r + \frac12}=0$.
Since $6(\frac15)^{1/3}-1>1/2$, the result clearly holds.

Suppose now that $\MM{x>r} = n$ and that the result holds for all
$\frac15\leq n' < n$. Let $u$ be the largest number for which
$\MM{x\ge r+u} > \frac{n}{5}$. As the distributions $\mu_i$ are
discrete, it follows that $\MM{x > r+u}\leq \frac{n}{5}$. As $u\ge
0$, we have $\mu_0\{x> r+u\}=0$. By the induction hypothesis with
$r$ replaced by $r+u$, we therefore get that
$$\MM{x > r + u+6(\frac{n}{5})^{1/3} -1}\;=\;0\;.$$

As $\mu_0\{x>r\}=0$, $\mm{x>r}\le n$ and $\MM{x\ge r+u} >
\frac{n}{5}$, we get by Theorem~\ref{T-asym}
that the sequence must
contain at least $\sqrt{3}(\frac15)^{3/2} (u-\frac{1}{2})^3 >
\frac17 (u-\frac{1}{2})^3$ moves whose centers are positive. As the
sequence of moves is constrained, and as $\mm{x>r}\le n$, there can
be at most~$n$ such moves with centers greater than $r$, i.e.,
$$\frac17 \left(u-\frac{1}{2}\right)^3 \;\le\; n\;.$$
Hence
$$u\leq (7n)^{1/3} + \frac{1}{2}\;,$$
and so
$$u+6(\frac{n}{5})^{1/3} - 1
  \;\le\; ( 7^{1/3} + 6\cdot (\frac15)^{1/3} ) n^{1/3} - \frac12
\;<\; 5.5 n^{1/3} - \frac12\;\leq\; 6 n^{1/3} - 1\;,$$ for $n\geq
1$.

This proves the induction step and completes the proof.
\end{proof}

Modulo the proofs of Lemmas~\ref{L-SM2} and~\ref{L-extreme}, which
are given in the next section, this completes the proof of our main
result that the maximum overhang that can be achieved using $n$
blocks in at most $6n^{1/3}$.
It is fairly straightforward to modify the proof of
Theorem~\ref{T-gen} above so as to obtain the stronger conclusion
that\linebreak $\MM{x>cn^{1/3}-1}=0$, for any $c>\frac{5^{5/2}}{2\cdot
3^{5/3}}\simeq 4.479$, at least for large enough values of~$n$, and
hence an improved upper bound on overhang of, say, $4.5n^{1/3}$.
This is done by choosing~$u$ to be the largest number for which
$\MM{x\ge u} > \frac{27}{125}n$. (The constant $\frac{27}{125}$ here
is the optimal choice.) The proof, however, becomes slightly
messier, as several of the inequalities do not hold for small values
of~$n$.

Next, we prove the following theorem which easily implies
Theorem~\ref{T-m2}.

\begin{theorem}\label{T-weightless} Let $\mu_0,\mu_1,\ldots,\mu_n$ be a sequence of
distributions obtained by applying a sequence of $n$ moves to an
initial distribution $\mu_0$ with $\mu_0\{x\le 0\}\le n$ and
$\mu_0\{x>0\}=0$, where $n\geq 2$. Then
$\mu_n\{x>2n^{1/3}\log_2n\}<1$.
\end{theorem}

\begin{proof} Suppose that $2^k\le n<2^{k+1}$, where $k\ge 0$. For
$1\le i\le k$, let $u_i$ be the largest number for which $\mm{x\ge
u_i}\ge \frac{n}{2^i}$. By the discreteness of the distributions we
again have $\mm{x> u_i}< \frac{n}{2^i}$. Let $u_0=0$. Assume, for
the sake of contradiction, that $\mu_n\{x>2n^{1/3}\log_2n\}\ge 1$.
Then, $u_k\ge 2kn^{1/3}$. There is then at least one value of~$i$
for which $u_i-u_{i-1}\ge 2n^{1/3}$. By Theorem~\ref{T-asym},
applied with $r=u_{i-1}$ and $d=u_i-u_{i-1}$, we conclude that the
sequence must contain more than~$n$ moves, a contradiction.
\end{proof}

As before, the constants in the above proof are not optimized. We
believe that a stronger version of the theorem, which states under
the same conditions that $\mu_n\{x>cn^{1/3}(\log_2 n)^{2/3}\}<1$,
for some $c>0$, actually holds. This would match an example supplied
by Johan H{\aa}stad. Theorem~\ref{T-weightless} (and
Theorem~\ref{T-m2}) imply an almost tight bound on an interesting
variant of the overhang problem that involves weightless blocks, as
discussed in Section~\ref{sec:concl}.

\subsection{Proof of spread vs. second moment inequalities}
\label{SS-proofs}



\begin{lemma-againa} {\rm(The proof was deferred from Section~\ref{SS-dist}.)}\ \
For any discrete distribution $\mu$,
$$S[\mu]^2 \;\le\; \frac{1}{3}\M_2[\mu] \M_0[\mu]^3\;.$$
\end{lemma-againa}

The method of proof used here was suggested to us by
Benjy Weiss, and resulted in a much improved and simplified
presentation. The lemma is essentially the case $n=2$ of a
more general result proved by Plackett~\cite{P47}.

\begin{proof}
Suppose that $\mu = \{ (x_1,m_1),...,(x_k,m_k) \}$
where $x_1< x_2 <\cdots <  x_k$.

We first transform the coordinates into a form which will be more convenient for applying the Cauchy-Schwartz inequality.
Since the statement of the lemma is invariant
under scaling of the masses,
we may assume that $\M_0[\mu]=1$.


Define a function $g(t)$ for $-\frac12\leq t\leq \frac12$ by
$$g(t) \eq x_i\;,\quad \mbox{where}
\quad \sum_{r=1}^{i-1}m_r < t+\frac12 \;\leq\; \sum_{r=1}^i m_r
\;,$$ and define $g(-\frac12) =x_1$.

Now we have that
$$M_j[\mu] \eq \sum_{i=1}^k x_i^j m_i \eq \tint{g(t)^j}$$
for $j\geq 0$, and
$$S[\mu] \eq \sum_{i<j} m_i m_j (x_j - x_i)
\eq \tint{\int_{s=-\frac12}^t \left(g(t)-g(s)\right)\;\d s} \;.$$
Above it may seem that the  integral should have been restricted
to the case where $g(s)<g(t)$. However, if $g(t)=g(s)$, the integrand is
zero, so this case does not contribute to the value of the integral.

Since $S$ is invariant under translation
and $\M_2$ is minimized by a translation which
moves $C[\mu]$ to the origin, we may assume without
loss of generality that $C[\mu]=0$, i.e., $M_1[\mu]=\tint{g(t)} = 0$.

Therefore
$$\tint{\int_{s=-\frac12}^t g(t)\;\d s} \eq \tint{\left(t+\frac12\right)g(t)}
\eq \tint{t g(t)} \;,$$ while
$$\tint{\int_{s=-\frac12}^t g(s)\; \d s}
\eq \int_{s=-\frac12}^\frac12  \int_{t=s}^\frac12 g(s)\;\d t \;\d s
\eq \int_{s=-\frac12}^\frac12 \left(\frac12-s\right) g(s)\;\d s \eq
- \int_{s=-\frac12}^\frac12 s g(s)\;\d s  \;.$$ So
$$S[\mu] \eq 2\tint{t g(t)} \;. \eqno{(\dagger)}$$

Using the Cauchy-Schwartz inequality,
$$S[\mu]^2 \eq 4\left(\tint{t g(t)}\right)^2 \;\leq\; 4\tint{g(t)^2}\cdot\tint{t^2}
\eq 4M_2[\mu]\cdot\frac1{12} \eq \frac13 M_2[\mu] \;.\vspace{-6ex}$$
\end{proof}
\vspace{3ex}

\renewcommand{\tint}[1]{\int_{-\frac12}^\frac12 {#1}\; \d t}

\begin{lemma-againb} {\rm(The proof was deferred from Section~\ref{SS-moves}.)}\ \
If $\mu_1$ is obtained from $\mu_0$ by an
extreme move (in an interval of length~1) then
$$ S[\mu_1]-S[\mu_0] \;\ge\; 3(\M_2[\mu_1]-\M_2[\mu_0])^2 \;.$$
\end{lemma-againb}

\begin{proof}
Since the statement of the lemma is invariant under linear
translation of the coordinates, we may assume that the interval of
the move is $[-\frac12,\frac12]$. Let $\nu_0 =
({\mu_0})_{[-\frac12,\frac12]}$, i.e., the restriction of $\mu_0$ to
$[-\frac12,\frac12]$.

Note that the lemma relates the difference in spread
and the difference in second moment resulting from the extreme
move. Since the addition of an extra point mass at either
$-\frac12$ or $\frac12$ leaves each of these differences invariant, we may
add such a mass as will bring the center of mass of $\nu_0$ to $0$,
and continue the proof under this assumption.
Since the statement of the lemma is invariant under scaling of the masses,
we may further assume that $M_0[\nu_0]=1$.

If $\nu_1$ is the result within the interval $[-\frac12,\frac12]$
of the extreme move, then:
$$\nu_1 \eq \left\{\left(-\frac12,\frac12\right), \left(\frac12,\frac12\right)\right\}
\mathrm{\quad and\quad }M_2[\nu_1] \eq S[\nu_1]\eq\frac14 \;.$$

We define $g(t)$ for $-\frac12\leq t\leq \frac12$ just as in the
proof of Lemma~\ref{L-SM2} but now corresponding to the
distribution~$\nu_0$,
and so $-\frac12 \le g(t) \le \frac12$, for $-\frac12\leq t\leq
\frac12$.
As before, $M_j[\nu_0]=\tint{g(t)^j}$ for $j\geq 0$, and we recall
as in~($\dagger$) that $S[\nu_0]= 2\tint{t g(t)}$.

We have $M_1[\nu_0]=\tint{g(t)}=0$. Let $c =
M_2[\nu_0]=\tint{g(t)^2}$ and $s = S[\nu_0]$.
By Lemma~\ref{L-SM2} we have $s^2\le \frac{c}{3}$. If $c\leq
\frac1{12}$ then $s\leq \sqrt{\frac{c}3}\leq \frac16$ and the result
follows immediately as
$$\textstyle S[\nu_1]-S[\nu_0] - 3(M_2[\nu_1]-M_2[\nu_0])^2 \eq
\frac14 - s - 3\left(\frac14 - c\right)^2 \;\geq\; \frac14 - s -
3\left(\frac14- 3s^2\right)^2 \eq \frac{(1+2s)(1-6s)^3}{16} \geq 0
\;.$$

We next claim that if $c=M_2[\nu_0]>\frac{1}{12}$, then
$s=S[\nu_0]\le \frac14 - \frac{a^2}{12}$, where $a = \frac32 - 6c
<1$. To prove this claim, we define a function $h(t)$ as follows:
$$
h(t)\eq \left\{ \begin{array}{cl}
 \frac{t}{a} & \mbox{if $|t| \leq \frac{a}2$, and}\\
 \frac12 \sgn(t) & \mbox{otherwise.}
\end{array}\right.
$$
We may verify that
$$\tint{h(t)^2} \eq \frac14-\frac{a}6 \eq c \quad\mathrm{and}\quad
  \tint{t\, h(t)} \eq \frac18 - \frac{a^2}{24} \;.$$

By the Cauchy-Schwartz inequality,
$$\left(\tint{h(t) g(t)}\right)^2 \;\leq\; \tint{h(t)^2}\cdot\tint{g(t)^2}
\eq c^2 ,$$ and so
$$\tint{h(t) g(t)} \;\leq\; c \eq \tint{h(t)^2} \;.\eqno{(*)}$$
We also have
$$\tint{\left(\frac{t}{a}-h(t)\right)g(t)}
\;\leq\; \tint{\left(\frac{t}{a}-h(t)\right)h(t)} \;,\eqno{(**)}$$
since $h(t)-g(t)\leq 0$ and $\frac{t}{a}-h(t) \leq 0$ for
$t<-\frac{a}2$, and $h(t)-g(t)\geq 0$ and $\frac{t}{a}-h(t) \geq 0$
for $t>\frac{a}2$, and $\frac{t}{a}-h(t) \eq 0$ for $|t|\leq
\frac{a}2$. Adding inequalities~(*) and~(**), and multiplying
by~$2a$, gives
$$S[\nu_0] \eq 2\tint{t\, g(t)} \;\leq\; 2\tint{t\, h(t)} \;=\; \frac14 - \frac{a^2}{12}\;.$$
Finally,
$$ S[\nu_1]-S[\nu_0] \;\ge\; \frac14 - \left( \frac14 - \frac{a^2}{12}
\right) \eq  \frac{a^2}{12} \eq 3\left( \frac14 - c \right)^2 \eq
3(\M_2[\nu_1]-\M_2[\nu_0])^2 \;.$$
This completes the proof.
\end{proof}

We end the section by noting that although the inequalities of
Lemmas~\ref{L-SM2} and~\ref{L-extreme} are only claimed for discrete
distributions, which is all we need in this paper, our proofs can be
easily modified to show that they hold also for general continuous
distributions. In fact, for non-trivial discrete distributions, the
inequalities in the two lemmas are always \emph{strict}. In the
continuous case, the inequalities are satisfied with equality by
appropriately chosen uniform distributions. In particular, the
constant factors $\frac13$ and $3$ appearing in the two lemmas
cannot be improved.

\section{Concluding remarks and open problems}\label{sec:concl}

We have shown that the maximum overhang achieved using~$n$
homogeneous, frictionless blocks of unit length is at most
$6n^{1/3}$.
Thus, the constructions of \cite{PZ06} cannot be improved by more
than a constant factor, establishing order $n^{1/3}$ as the
asymptotic answer to the age-old overhang problem.

The discussions and results presented so far all referred to the
standard two-dimensional version of the overhang problem. Our
results hold, however, in greater generality. We briefly discuss
some natural generalizations and variants of the overhang problem
for which our bounds still apply.

In Section~\ref{sec:model} we stipulated that all blocks have a
given height~$h$. It is easy to see, however, that all our results
remain valid even if blocks have different heights, but still have
unit length and unit weight. In particular, blocks are allowed to
degenerate into \emph{sticks}, i.e., have height~$0$. Also, even
though we required blocks not to overlap, we did not use this
condition in any of our proofs.

\emph{Loaded stacks}, introduced in~\cite{PZ06}, are stacks composed
of standard unit length and unit weight blocks, and \emph{point
weights} that can have arbitrary weight. (Point weights may be
considered to be blocks of zero height and length, but nonzero
weight.) Our results, with essentially no change, imply that loaded
stacks of total weight~$n$ can have an overhang of at most
$6n^{1/3}$. 

What happens when we are allowed to use blocks of different lengths
and weights? Our results can be generalized in a fairly
straightforward way to show that if a block of length~$\ell$ has
weight proportional to~$\ell^3$, as would be the case if all blocks
were similar three-dimensional cuboids, then the overhang of a stack
of total weight~$n$ is again of order at most $n^{1/3}$. It is
amusing to note that in this case an overhang of order~$n^{1/3}$ can
be obtained by stacking $n$ unit-length blocks as in the
construction of~\cite{PZ06}, or simply by balancing a single block
of length~$n^{1/3}$ and weight~$n$ at the edge of the table!  Might
this mean that there is some physical principle that could have told
us, without all the calculations, that the right answer to the
original overhang problem had to be of order $n^{1/3}$?


Theorem~\ref{T-m2} supplies an almost tight upper bound for the
following variant of the overhang problem: How far away from the
edge of a table can a mass of weight~1 be supported using~$n$
\emph{weightless} blocks of length~$1$, and a collection of point
weights of total weight $n$? The overhang in this case beats the
classical one by a factor of between $\log^{2/3} n$ and $\log n$.

In all variants considered so far, blocks were assumed to have their
largest faces parallel to the table's surface and perpendicular to
its edge. The assumption of no friction then immediately implied
that all forces within a stack are vertical, and our analysis, which
assumes that there are no horizontal forces, was applicable. A nice
argument, communicated to us by Harry Paterson, shows that in the
frictionless two-dimensional case, horizontal forces cannot be
present even if some of the blocks are tilted. Our results thus
apply also in this case.

\begin{figure}[th]
\begin{center}
\includegraphics[width=80mm]{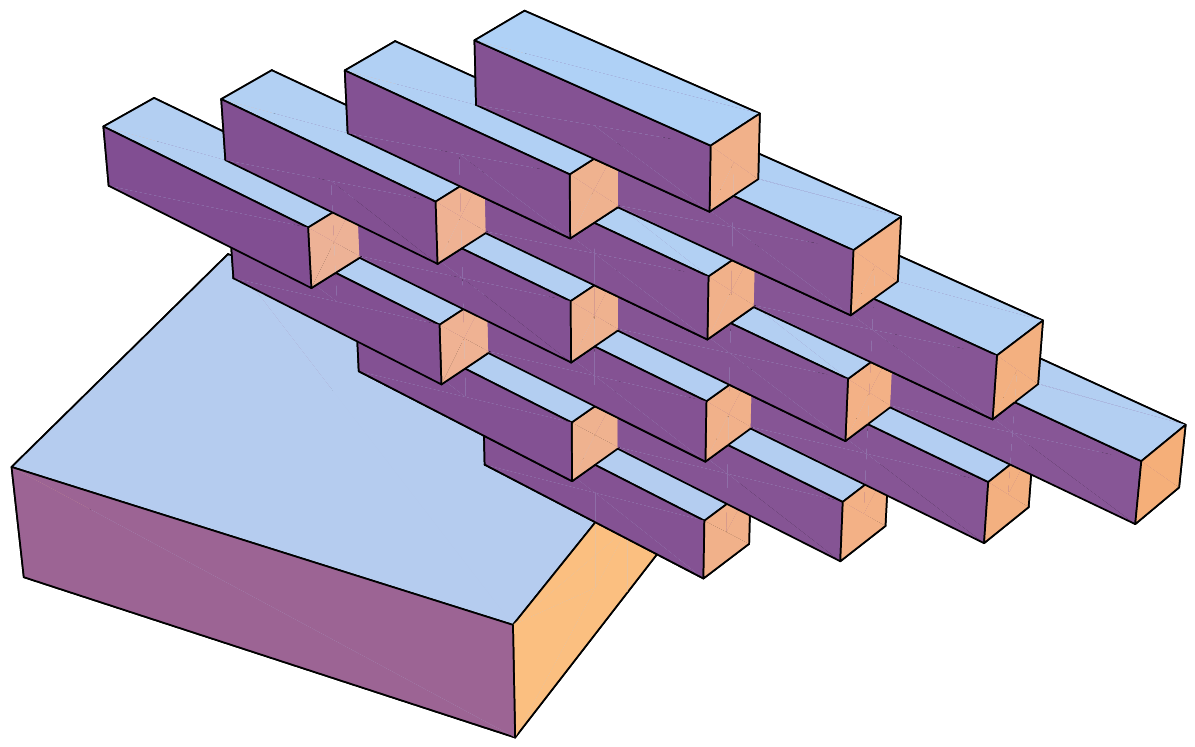}
\caption{A ``skintled'' $4$-diamond.} \label{fig:skintling}
\end{center}\vspace{-5mm}
\end{figure}

We believe that our bounds apply, with slightly adjusted constants,
also in three dimensions, but proving so remains an open problem.
Overhang larger by a factor of $\sqrt{1+w^2}$ may be obtained with
$1\times w\times h$ blocks, where $h\le w\le 1$, using a technique
called \emph{skintling} (see Figure~\ref{fig:skintling}). In
skintling (a term we learned from an edifying conversation with John
H. Conway about brick-laying) each block is rotated about its
vertical axis, so that---in our case---the diagonal of its bottom
face is perpendicular to the edge of the table.  With suitably
adjusted notion of length, however, our bounds apply to any
three-dimensional construction that can be balanced using vertical
forces. It is an interesting open problem whether there exist
three-dimensional stacks composed of frictionless, possibly tilted,
blocks that can only be balanced with the aid of some non-vertical
forces. (We know that this is possible if the blocks are
\emph{nonhomogeneous} and are of different sizes.) As mentioned, we
believe that our bounds do apply in three dimensions, even if it
turns out that non-vertical forces are sometimes useful, but proving
this requires some additional arguments.

We end by commenting on the tightness of the analysis presented in
this paper. Our main result is a~$6n^{1/3}$
upper bound on the overhang that may be obtained using~$n$ blocks.
As mentioned after the proof of Theorem~\ref{T-gen}, this bound can
be easily improved to about $4.5n^{1/3}$, for sufficiently large
values of~$n$. Various other small improvements in the constants are
possible. For example, a careful examination of our proofs reveals
that whenever we apply Lemma~\ref{L-extreme}, the distribution
$\mu_0$ contains at most three masses in the interval acted upon by
the move that produces~$\mu_1$. (This follows from the fact that a
block can rest upon at most three other blocks.) The constant~$3$
appearing in Lemma~\ref{L-extreme} can then be improved, though it
is optimal when no assumption regarding the distribution~$\mu_0$ is
made. We believe, however, that new ideas would be needed to reduce
the upper bound to below, say, $3n^{1/3}$.

\begin{figure}[th]
\begin{center}
\includegraphics[width=100mm]{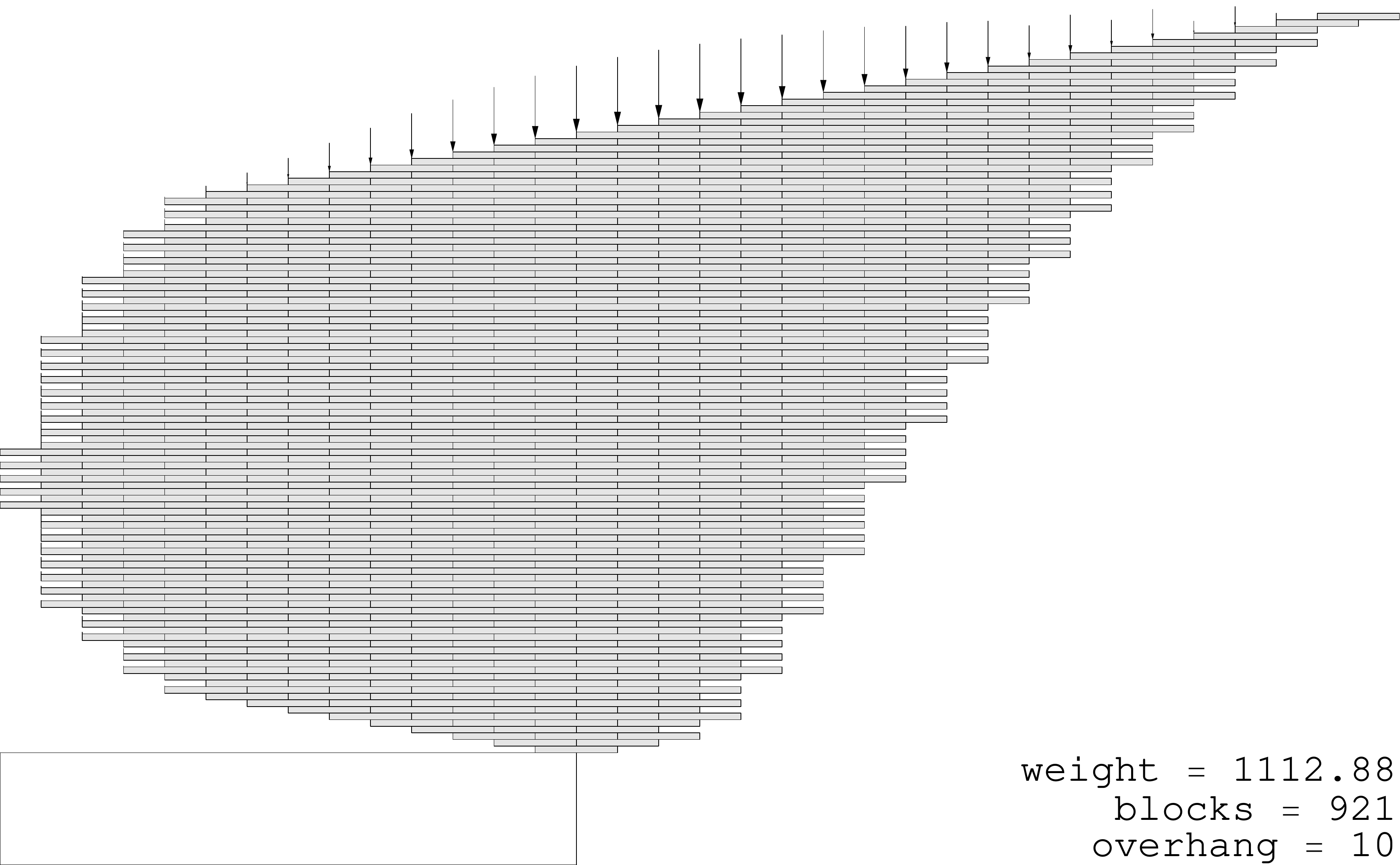}
\caption{An ``oil-lamp''-shaped stack} \label{fig:oil-lamp}
\end{center}\vspace{-5mm}
\end{figure}

As mentioned, Paterson and Zwick~\cite{PZ06} describe simple
balanced $n$-block stacks that achieve an overhang of about
$0.57n^{1/3}$. They also present some numerical evidence that
suggests that the overhang that can be achieved using~$n$ blocks,
for large values of~$n$, is at least $1.02n^{1/3}$.
These larger overhangs are obtained using stacks that are shaped
like the ``oil-lamp'' depicted in Figure~\ref{fig:oil-lamp}. For
more details on the figure and on ``oil-lamp'' constructions, see
\cite{PZ06}. (The stack shown in the figure is actually a loaded
stack, as defined above, with the external forces shown representing
the point weights.)

A small gap still remains between the best upper and lower bounds
currently available for the overhang problem, though they are both
of order $n^{1/3}$. Determining a constant~$c$ such that the maximum
overhang achievable using~$n$ blocks is asymptotically
$cn^{1/3}$ is a challenging open problem.


\section*{Acknowledgements} We would like to thank John H. Conway, Johan H{\aa}stad, Harry Paterson, Anders
Thorup and Benjy Weiss for useful discussions
observations, some of which appear with due credit within the
paper.

\end{document}